\documentclass[12pt]{amsart}
\usepackage{amsfonts,amssymb,amscd,xy}
\usepackage[leqno]{amsmath}
\usepackage{mathrsfs}
\usepackage{amssymb}
\usepackage{amsmath}
\usepackage{amsxtra}
\usepackage{amsthm}  
\usepackage{xcolor}

\usepackage[colorlinks=true]{hyperref}

\def\lbe{\kern -.025em}

\newcommand{\gr}{{\mathrm{Gr}\lbe}}

\newcommand{\sep}{{\rm sep}}

\newcommand{\id}{\mathrm{id}}
\newcommand{\et}{{\rm {\acute et}}}

\newcommand{\fppf}{{\rm fppf}}

\newcommand{\pf}{{\rm pf}}

\newcommand{\A}{{\mathbb A}}
\newcommand{\Z}{{\mathbb Z}}
\newcommand{\Q}{{\mathbb Q}}
\newcommand{\N}{{\mathbb N}}
\newcommand{\F}{{\mathbb F}}
\newcommand{\G}{{\mathbb G}}

\newcommand{\spec}{\mathrm{Spec}\,}

\newcommand{\Hom}{{\mathrm{Hom}}}

\newcommand{\ff}{{\mathbf F}}
\newcommand{\vv}{{\mathbf V}}

\newcommand{\rpsch}{{\mathrm{RPSch}}}
\newcommand{\rpaff}{{\mathrm{RPAff}}}
\newcommand{\rp}{{\mathrm{RP}}}

\newcommand{\cO}{{\mathcal O}}

\DeclareMathOperator{\Tor}{Tor}
\DeclareMathOperator{\Ext}{Ext}

\numberwithin{equation}{section}

\newtheorem{lemma}[equation]{Lemma} 
\newtheorem{theorem}[equation]{Theorem}
\newtheorem{proposition-definition}[equation]{Proposition-Definition}

\newtheorem{proposition}[equation]{Proposition}
\newtheorem{definition}[equation]{Definition}
\theoremstyle{definition}

\theoremstyle{remark}
\newtheorem{remark}[equation]{Remark}

\newtheorem{example}[equation]{Example}

\newtheorem*{acknowledgments}{Acknowledgments}
\newtheorem*{notation}{Notation and Conventions}

\newcommand{\Weil}{\mathfrak{R}} 
\newcommand{\Order}{\mathcal{O}}
\newcommand{\into}{\hookrightarrow}
\newcommand{\onto}{\twoheadrightarrow}
\newcommand{\isomto}{\overset{\sim}{\to}}

\newcommand{\compose}{\mathbin{\circ}}
\newcommand{\tensor}{\mathbin{\otimes}}
\newcommand{\closure}[1]{\overline{#1}}
\newcommand{\dirlim}{\varinjlim}
\newcommand{\invlim}{\varprojlim}
\newcommand{\ur}{\mathrm{ur}}
\newcommand{\ideal}[1]{\mathfrak{#1}}
\newcommand{\RPS}{\mathrm{RPS}}
\newcommand{\RP}{\mathrm{RP}}
\newcommand{\Et}{\mathrm{Et}}
\newcommand{\RPSch}{\mathrm{RPSch}}
\newcommand{\Sch}{\mathrm{Sch}}
\newcommand{\Gr}{\mathrm{Gr}}
\newcommand{\RPSSch}{\mathrm{RPSSch}}
\newcommand{\cl}{\mathrm{cl}}
\newcommand{\Ver}{\mathrm{\mathbf{V}}}
\newcommand{\Fr}{\mathrm{\mathbf{F}}}
\newcommand{\Affine}{\mathbb{A}}
\newcommand{\Gm}{\mathbb{G}_{m}}
\newcommand{\Ga}{\mathbb{G}_{a}}
\DeclareMathOperator{\Spec}{Spec}
\DeclareMathOperator{\Ab}{Ab}
\DeclareMathOperator{\Ker}{Ker}
\DeclareMathOperator{\Coker}{Coker}
\DeclareMathOperator{\Gal}{Gal}

\begin{document}

\input xy     
\xyoption{all}

\title[The relatively perfect Greenberg transform]{The relatively perfect Greenberg transform and cycle class maps}

\author{Alessandra Bertapelle}
\address{Universit\`a degli Studi di Padova, Dipartimento di Matematica, via Trieste 63, I-35121 Padova, Italia}
\email{alessandra.bertapelle@unipd.it}

\author{Takashi Suzuki}
\address{Department of Mathematics, Chuo University,
1-13-27 Kasuga, Bunkyo-ku, Tokyo 112-8551, Japan}
\email{tsuzuki@gug.math.chuo-u.ac.jp}
\thanks{The second author was a Research Fellow of Japan Society for the Promotion of Science and supported by JSPS KAKENHI Grant Number JP18J00415.}
\date{May 24, 2024}


\begin{abstract}
Given a scheme over a complete discrete valuation ring of mixed characteristic with perfect residue field, the Greenberg transform produces a new scheme over the residue field thicker than the special fiber.
In this paper, we will generalize this transform to the case of imperfect residue field.
We will then construct a certain kind of cycle class map defined on this generalized Greenberg transform applied to the N\'eron model of a semi-abelian variety, which takes values in the relatively perfect nearby cycle functor defined by Kato and the second author.
\end{abstract}

\maketitle

\tableofcontents


\section{Introduction}


\subsection{Aim of the paper}
The \emph{Greenberg transform} (also called \emph{Greenberg realization})
is a process that, given any scheme $X$ over a complete discrete valuation ring $\Order_{K}$, defines a canonical $k$-scheme structure on the set
$X(\Order_{X})$ of $\Order_{K}$-valued points, where $k$ is the residue field of $\Order_{K}$ assumed to be perfect of characteristic $p > 0$.
This transform (or functor) is originally defined by Greenberg \cite{Gre61}
and applied in many situations (\cite{Beg81}, \cite{BT14}, \cite{CLNS18}, \cite{Lip76}, \cite{NS08}, \cite{Ser61}, \cite{Suz20} for example).
A modern and thorough account on its foundation can be found in \cite{BGA18Gre}.
Our first goal in this paper is to generalize this transform
to the case where the residue field $k$ is not necessarily perfect
but has $[k : k^{p}] < \infty$.

Complete discrete valuation rings with such residue field show up in geometry.
A typical example is the completed local ring of a normal scheme flat of finite type over $\Z_{p}$ at the generic point of an irreducible component of the special fiber.
An explicit example is the $p$-adic completion of the local ring $\Z_{p}[t]_{(p)}$,
which is absolutely unramified with residue field $\F_{p}(t)$.
A more ``local'' type of example is the ring of integers of a higher-dimensional local field
(\cite{Kat79}, \cite{Par84}, \cite{FK00}).

The main difficulty with imperfect $k$ is the poor behavior of
the ring $W(k)$ of $p$-typical Witt vectors of $k$.
In particular, $\Order_{K}$ does not admit a canonical algebra structure over $W(k)$.
This makes it hard to utilize the functorial property of $W$ to define a functor
that a Greenberg transform should represent.

In this paper, we deal with this problem by defining our Greenberg transform
as a \emph{relatively perfect scheme} over $k$.
A relatively perfect $k$-scheme $Y$ is a $k$-scheme such that
the relative Frobenius morphism $Y \to Y^{(p)}$ over $k$ is an isomorphism
(\cite{Kat86}, \cite{Kat87}).
When $k$ is perfect, this agrees with the usual definition of perfect schemes,
and the perfect scheme version of the Greenberg transform
already appears in the literature
(see \cite{Beg81}, \cite[Section 12]{BGA18Gre} for example).
In several situations, what matters is the perfection of the Greenberg transform
rather than the Greenberg transform itself;
compare \cite{Ser61} with \cite[Appendix]{DG70} for example.
For general $k$, using Kato's notion of
\emph{canonical lifting} \cite[Definition 1]{Kat82}
of a relatively perfect $k$-scheme,
we will define our \emph{relatively perfect Greenberg transform}
and prove its existence.

Assuming the fraction field $K$ of $\Order_{K}$ has
characteristic $0$,  
our next goal is to apply our relatively perfect Greenberg transform to define
a certain kind of \emph{cycle class map}.
Our cycle class map takes values in the \emph{relatively perfect nearby cycle} $R^{m} \Psi$
defined in \cite{KS19}.
This nearby cycle functor is $p$-adic in nature,
while our Greenberg transform is more algebraic.
The Greenberg transform provides an object
``before $p$-adic completion'' of the $p$-adic nearby cycle.
We will define a canonical morphism from the relatively perfect Greenberg transform
of the N\'eron model of a semi-abelian variety $G$ over $K$
to $R^{1} \Psi(T_{p} G)$, where $T_{p} G$ is the $p$-adic Tate module of $G$.
We call this morphism the \emph{relatively perfect cycle class map} for $G$ in degree $1$.
We will show that its $p$-adic completion is a closed immersion for general $G$
and an isomorphism for any torus $G$.

Of course we want to treat more general motives than just semi-abelian varieties
and higher degree nearby cycles $R^{m} \Psi$ with $m > 1$.
But we do not have such a generalization at present.
From this motivic viewpoint,
our relatively perfect Greenberg transform should be considered as
an example of an ``\'etale motivic nearby cycle with $p$-torsion'';
compare this with the picture presented in the second author's article in \cite[pp.1747--1750, Section 3]{GHHKK19}.

We will develop a solid foundation for relatively perfect group schemes
in order to construct the relatively perfect cycle class map.
This foundational material is also applied to
arithmetic duality for two-dimensional local rings \cite{Suz22}
and existence of N\'eron models \cite{OS23}.


\subsection{Main results}
Now we formulate (a special case of) our results.
Let $K$ be a complete discrete valuation field with ring of integers $\Order_{K}$
and residue field $k$ such that $k$ has characteristic $p > 0$ and $[k : k^{p}] < \infty$.
For a relatively perfect $k$-algebra $R$, its Kato canonical lifting \cite[Definition 1]{Kat82} is a unique complete flat $\Order_{K}$-algebra $h(R)$
together with an isomorphism $h(R) \otimes_{\Order_{K}} k \cong R$.
In particular, $h(k)$ is just $\Order_{K}$.

\begin{theorem} \label{thm: Greenberg transform}
	Let $X$ be an $\Order_{K}$-scheme.
	Consider the functor
		\[
				F_{X}
			\colon
				R
			\mapsto
				X(h(R))
		\]
from the category of relatively perfect $k$-algebras to the category of sets.
This functor is representable by a relatively perfect $k$-scheme.
We call the object representing $F_{X}$ the \emph{relatively perfect Greenberg transform} of $X$ and denote it by $\gr^{\rp}(X)$.
\end{theorem}

In particular, $\Gr^{\RP}(X)(k) = X(\Order_{K})$.
Hence $\Gr^{\RP}(X)$ is a canonical relatively perfect $k$-scheme structure on $X(\Order_{K})$.

Actually, we prove more precise and general results.
The ring $\Order_{K}$ may be replaced by any complete noetherian local ring
with residue field $k$ or $W(k)$.
Also, for any $m > 0$, we will give a ``finite level'' relatively perfect Greenberg transform of an $\Order_{K} / \ideal{p}_{K}^{m}$-scheme $X$,
where $\ideal{p}_{K}$ is the maximal ideal of $\Order_{K}$.
View this construction as associated with the datum
$\Spec k \into \Spec \Order_{K} / \ideal{p}_{K}^{m}$.
This datum may be generalized to a certain more general nilpotent thickening $T \into S$ of not necessarily affine schemes (see the beginning of Sections \ref{s.gf} and \ref{s.existence} for the precise conditions needed).
In this setting, we will define a relatively perfect Greenberg transform
that assigns to any $S$-scheme a relatively perfect $T$-scheme.

We also define a certain finite level Greenberg transform ``before relative perfection'' associated with the datum $\Spec k \into \Spec \Order_{K} / \ideal{p}_{K}^{m}$.
Its relative perfection (see below) recovers the finite level relatively perfect Greenberg transform.
This should serve the needs of those who are unfamiliar with relatively perfect schemes and/or concerned with the information lost by relative perfection (such as nilpotents).
However, this construction has a problem:
it does not recover the usual Greenberg transform even if $k$ is perfect;
and for general $k$, its natural inverse limit in $m$ is automatically relatively perfect.
In this sense, the relatively perfect Greenberg transform still seems to be
the only possible construction in the infinite level, imperfect residue field situation.

Before the next theorem, we recall the relatively perfect nearby cycle functor $R^{m} \Psi$ from \cite{KS19}.
By \cite[Proposition 1.4]{Kat86}, the inclusion functor from the category of relatively perfect $k$-schemes to the category of all $k$-schemes admits a right adjoint $Y \mapsto Y^{\rp}$.
The scheme $Y^{\rp}$ is called the \emph{relative perfection} of $Y$.
We say that a $k$-scheme $Y$ is \emph{relatively perfectly smooth} (\cite[Section 2]{KS19})
if it is Zariski locally isomorphic to the relative perfection of a smooth $k$-scheme.
The category of relatively perfectly smooth $k$-schemes endowed with the \'etale Grothendieck topology is a site, which we denote by $\Spec k_{\RPS}$.
We denote by $\Ab(k_{\RPS})$ the category of sheaves of abelian groups on $\Spec k_{\RPS}$.
Let $\Spec K_{\Et}$ be the category of $K$-schemes endowed with the \'etale topology
and $\Ab(K_{\Et})$ its category of sheaves of abelian groups.
Assume that $K$ has characteristic $0$.
For any $m \ge 0$, by \cite[Corollary 3.3]{KS19}, the functor
	\[
			R^{m} \Psi
		\colon
			\Ab(K_{\Et})
		\to
			\Ab(k_{\RPS})
	\]
is given by sending $F \in \Ab(K_{\Et})$ to the \'etale sheafification of the presheaf
	\[
			R
		\mapsto
			H^{m}(h(R) \otimes_{\Order_{K}} K, F),
	\]
where $R$ runs through relatively perfectly smooth $k$-algebras.
All the N\'eron models in this paper are N\'eron lft (locally finite type) models
in the terminology of \cite[Chapter 10]{BLR90}.

\begin{theorem} \label{thm: RP cycle class map}
Let $G$ be a semi-abelian variety over $K$ with N\'eron model $\mathcal{G}$ over $\Order_{K}$.
	Denote the kernel of the multiplication-by-$p^{n}$ morphism on $G$ by $G[p^{n}]$.
	Let $N$ be a $p$-primary finite Galois module over $K$.
	\begin{enumerate}
		\item \label{item: representability of nearby cycle}
			For any $m \ge 0$, the sheaf $R^{m} \Psi N \in \Ab(k_{\RPS})$
			is represented by an affine relatively perfectly smooth $k$-scheme.
		\item \label{item: existence of cycle class map}
			Set $R^{1} \Psi(T_{p} G) = \invlim_{n} R^{1} \Psi(G[p^{n}])$.
			Then there exists a canonical morphism
				\[
						\cl_{p}
					\colon
						\Gr^{\RP}(\mathcal{G})
					\to
						R^{1} \Psi(T_{p} G)
				\]
                of relatively perfect group schemes over $k$, which we call the \emph{relatively perfect cycle class map} for $G$ in degree $1$.
		\item \label{item: points of cycle class map}
			Let $k^{\sep}$ be a separable closure of $k$.
			Let $\Hat{K}^{\ur}$ be the corresponding completed unramified extension of $K$.
			Then the map induced by $\cl_{p}$ on $k^{\sep}$-valued points is the canonical map
				\[
						G(\Hat{K}^{\ur})
					\to
						H^{1}(\Hat{K}^{\ur}, T_{p} G)
				\]
			given by the coboundary map of the Kummer sequence,
			where the right-hand side is the continuous cohomology of the $p$-adic Tate module $T_{p} G$.
		\item \label{item: inj and isom for cycle class map}
			The morphism
				\[
						\cl_{p}^{\wedge}
					\colon
						\Gr^{\RP}(\mathcal{G})^{\wedge_{p}}
					\to
						R^{1} \Psi(T_{p} G)
				\]
			induced by $\cl_{p}$ via $p$-adic completion is
			a closed immersion.
			It is an isomorphism if $G$ is a torus.
	\end{enumerate}
\end{theorem}

Statements \eqref{item: existence of cycle class map} and \eqref{item: points of cycle class map} say that the coboundary map of the Kummer sequence preserves
the canonical relatively perfect scheme structures over the residue field $k$.

Statement \eqref{item: cycle class map is injective} is essentially proved in \cite[Proposition 6.1]{KS19} when $N$ is a Tate twist of $\Z / p \Z$.
We need to generalize it to include the case of $N = G[p^{n}]$ for all $G$ and $n$.
The minimal extension of $K$ that trivializes the Galois action on $N$ can have inseparable residue field extension.
We will deal with this by Weil restrictions.


\subsection{Organization}
In Section \ref{s.rp},
we recall relatively perfect schemes and relative perfection
and prove a key proposition (Proposition \ref{p.wh})
about a nice behavior of the Witt vector functor applied to relatively perfect algebras.
In Section \ref{s.cl},
we recall the Kato canonical lifting $h(Q)$ of a relatively perfect algebra $Q$
and generalize it to the non-affine case.
In Section \ref{s.gf},
we discuss sheaf-theoretic properties of the finite-level version of the functor $F_{X}$ (and its generalizations)
that appears in Theorem \ref{thm: Greenberg transform}.
In Section \ref{s.existence},
we prove the representability of the finite level $F_{X}$ in the cases of interest
and give some properties of the representing objects when $X$ is smooth.
In Section \ref{sec: Greenberg transform of infinite level},
we prove the representability of the infinite level $F_{X}$
as stated in Theorem \ref{thm: Greenberg transform}.
In Section \ref{sec: Greenberg transform before relative perfection},
we discuss a Greenberg transform before relative perfection
and its relation to the relatively perfect Greenberg transform.
In Section \ref{sec: Relatively perfectly smooth group schemes},
we give some foundational results about relative perfections of smooth group schemes over $k$.
In Section \ref{sec: Relatively perfect cycle class maps},
we construct the relatively perfect cycle class map and prove its properties.
Then Theorem \ref{thm: RP cycle class map} \eqref{item: representability of nearby cycle} follows from Proposition \ref{prop: nearby cycle is RP algebraic group} and the other assertions in the theorem are immediate consequences of Definition \ref{def: RP p-adic cycle class map} and  Propositions \ref{prop: Gr to nearby cycle} and \ref{prop: cycle class map for tori}.

\begin{acknowledgments}
	This work started from the second author's visit to the first author
	at the University of Padua in 2018.
	The second author is grateful for the hospitality and the excellent work environment
	provided by the University of Padua resulting this collaboration.
        Also, the authors would like to thank Otto Overkamp
        for pointing out that quasi-separateness is needed in
        Definition \ref{def: finiteness for relatively perfect schemes}
        and the referee for careful comments.
\end{acknowledgments}

\begin{notation}
	We fix a prime number $p$.
	We denote by $W_N$ the $\F_p$-group scheme of $p$-typical Witt vectors of length $N$.
	Rings are tacitly assumed to be associative and commutative with unit, and ring homomorphisms send unit to unit.
\end{notation}


\section{Relative perfection}
\label{s.rp}
Let $T$ be an $\F_p$-scheme. Recall that $T$ is said to be {\it perfect} (respectively, {\it semi-perfect}) if the absolute Frobenius $\ff_{T}\colon T\to T$ is an isomorphism (respectively, a closed immersion). More generally, a scheme $Y$ over $T$ is said to be {\it relatively perfect} (respectively, {\it relatively semi-perfect}) {\it over T} if the canonical morphism $\ff_{Y/T}\colon Y\to Y^{(p)}:=Y\times_{T,\ff_{T}} T$ (the relative Frobenius morphism) induced by the absolute Frobenius $\ff_{Y}$ is an isomorphism (respectively, a closed immersion).
\[\xymatrix{Y\ar@/^1pc/[rr]^{\ff_{Y}}\ar[r]_(0.4){\ff_{Y/T}}\ar[d] &Y^{(p)}\ar[d] \ar@{}[dr] | {\square} \ar[r]   &Y\ar[d] \\
T\ar@{=}[r]&T\ar[r]^{\ff_{T}}&T
}\] 
 Note that $Y$ is relatively perfect over $T$ if and only if $\ff_{Y}$ is the base change of $\ff_{T}$. 
 If $T$ is perfect, then any relatively perfect  scheme over $T$ is perfect and conversely.
Note that the notation $Y^{(p)}$ is not ambiguous as long as $T$ is fixed. 

Let $(\mathrm{Sch}/T)$ denote the category of $T$-schemes and let $(\rpsch/T)$ be the full subcategory of  $(\mathrm{Sch}/T)$ consisting of all relatively perfect schemes over $T$.
Let $(\rpaff / T)$ be its full subcategory consisting of affine schemes.
Note that if $Y$ is perfect and $T$ is  semi-perfect,  then $Y$ is relatively perfect over $T$. However there exist  perfect schemes $Y$  not relatively perfect over $T$:
for example, consider the field of rational functions in one variable $\F_p(x)$ and its algebraic closure $\F_p(x)^a$; then    $ Y=\spec \F_p(x)^a$ is perfect but not relatively perfect over $ T=\spec \F_p(x)$. More general counterexamples can be constructed using the fact that if   $Y$ is relatively perfect over $T$ then $Y\to T$ is formally \'etale \cite[Lemma 1.3]{Kat86}.  

\begin{proposition}\label{p.rp}  Assume that $\ff_{T}$ is a finite, locally free morphism.  Then the inclusion functor $(\rpsch/T)\to (\mathrm{Sch}/T) $ has a right adjoint $(\ )^\rp\colon  (\mathrm{Sch}/T)\to 	(\rpsch/T)$.
\end{proposition}
This proposition is proved in \cite[Proposition 1.4]{Kat86} and the relative perfect scheme $Y^\rp$ associated with a $T$-scheme $Y$ is called the {\it relative perfection of $Y$ over $T$}. 

\begin{remark} \label{rem: relative perfectness and Frobenius}
\begin{enumerate}
    \item 
    Note that the absolute Frobenius $\ff_{T}$ is finite and locally free    when $T$ is smooth over a perfect $\F_p$-scheme $T'$. This fact is immediate to prove if $T=\A^r_{T'}$. In the general case, locally on $T$ and $T'$, one may assume that $T$ is \'etale over $\A^r_{T'}$; then $\ff_{T}$ is   the base change of  $\ff_{\A^r_{T'}}$ \cite[Lemma 1.3]{Kat86}, thus the conclusion follows.
    \item \label{item: relatively perfect implies geom reduced}
	Take $T$ to be the spectrum of a field $k$.
	Then a relatively perfect $k$-scheme is geometrically reduced.
    Indeed, a relatively perfect $k$-algebra $R$ is reduced since
	$\ff_{R} \colon R \into R \otimes_{k, \ff_{k}} k \isomto R$
	is injective, where the first homomorphism is $\id \otimes 1$
	and the second homomorphism is
	$\ff_{R} \otimes \mathrm{can}$.
	Therefore $R \otimes_{k, \ff_{k}} k \isomto R$ is reduced.
	This implies that $R \otimes_{k} \closure{k}$ is reduced,
	so $R$ is geometrically reduced.
\end{enumerate}
\end{remark}

We recall here below some details on the proof of Proposition \ref{p.rp}.   From \cite[Lemma 1.5]{Kat86} one has:
\begin{lemma} \label{l.FrWeil}
Assume that $\ff_{T}$ is a finite locally free morphism. Then the base change functor along the absolute Frobenius $(\mathrm{Sch}/T)\to (\mathrm{Sch}/T), Y\mapsto Y^{(p)}$ has a right adjoint.
\end{lemma}

The above lemma asserts that for any $T$-scheme $Y$, the  Weil restriction  of $Y$ along the absolute Frobenius $\ff_{T}$  exists; a generalization of this fact to any finite and locally free universal homeomorphism was proved in \cite[Corollary A.5]{BGA18Gre}.  Kato denotes the Weil restriction of $Y$ along $\ff_{T}$ as  $G(Y)$; we prefer to write  $\Weil_{\ff_{T}}(Y)$, thus keeping notation similar to the one in  \cite[Section 7.6]{BLR90}.
 The Weil restriction of $Y$ along $\ff_{T}$ comes equipped with a canonical  morphism of $T$-schemes \[\rho_Y\colon 
 \Weil_{\ff_{T}}(Y)^{(p)} \longrightarrow Y\] such that the map
 \[\Hom_{T}(Y',\Weil_{\ff_{T}}(Y))\to \Hom_{T}(Y'^{(p)}, Y), \quad u\mapsto \rho_Y\circ u^{(p)}, \]
 is a bijection for any $T$-scheme $Y'$, where $u^{(p)}\colon Y'^{(p)}\to \Weil_{\ff_{T}}(Y)^{(p)}$ is the base change of $u$ along $\ff_{T}$. Note that the morphism $\rho_Y$ is associated with the identity morphism on $\Weil_{\ff_{T}}(Y)$  via the above natural bijection and its base change along any open immersion $U\to Y$ is $\rho_U$. Further $\Weil_{\ff_{T}}(U)$ is affine if $U$ is affine. 
Let \[g_{Y/T}:=\rho_{Y}\circ \ff_{\Weil_{\ff_{T}}(Y)/T}\colon \Weil_{\ff_{T}}(Y)\to Y. \]
By construction $g_{Y/T}$ is an affine morphism. Indeed $\rho_{Y}$ is affine and the relative Frobenius $ \ff_{\Weil_{\ff_{T}}(Y) / T} \colon \Weil_{\ff_{T}}(Y)\to \Weil_{\ff_{T}}(Y)^{(p)}$  is a universal homeomorphism,  in particular affine.
Hence the following inverse limit of $T$-schemes
\[Y^\rp:=\varprojlim (Y\stackrel{g_{Y/T}}{\longleftarrow} \Weil_{\ff_{T}}(Y)\stackrel{g_{\Weil_{\ff_{T}}(Y)/T}}{\longleftarrow}    \Weil_{\ff_{T}}(\Weil_{\ff_{T}}(Y)) \longleftarrow \dots)=\varprojlim_{n\in \N_0} \Weil_{\ff_{T}}^{n}(Y)\]
is a scheme, where $\Weil_{\ff_{T}}^{n}$ denotes the $n$-th iterate of the Weil restriction construction. Further, $Y^\rp$  is a relatively perfect scheme over $T$.

\begin{remark}\label{rem: properties of RP} \begin{enumerate}
\item  If $T$ is a perfect scheme, $\Weil_{\ff_{T}}(Y)=Y^{( \frac 1p)}:= Y\times_{T, \ff^{-1}_{T}}\!T$, $\rho_{Y}=\id_{Y}$ and $g_{\Weil_{\ff_{T}}}\!(Y)\colon Y^{(\frac 1p)}\to Y$ is the relative Frobenius so that $Y^\rp$ is the (inverse) perfection $Y^\pf$ of $Y$. Recall that $Y^\pf$ is also the inverse limit of numerable copies of $Y$ with $\ff_{Y}$ as transition maps \cite[Section 5]{BGA18Per}.  
\item \label{RP and limits}  Since relative perfection is right adjoint to the inclusion functor $(\rpsch/T)\to (\mathrm{Sch}/Y)$, it preserves limits and thus, if $Y$ is a group scheme over $T$, the same is $Y^\rp$ and the canonical map $Y^{\rp}\to Y$ is a morphism of group schemes over $T$.
\item \label{RP preserves qc} Open immersions and more generally \'etale morphisms are preserved by relative perfection \cite[Corollary 1.9]{Kat86}. Further, if $\ff_{T}$ is finite locally free, any property of morphisms of schemes that is preserved by Weil restriction and inverse limits  is preserved by relative perfection. For example, if  $f$ is a morphism of $T$-schemes that is separated or quasi-compact (with $T$ locally noetherian)  or a closed immersion, the same is $f^\rp$: see \cite[Section 7.6, Proposition 2 and Proposition 5]{BLR90} and \cite[Proposition 3.2]{BGA18Per}).

\end{enumerate}

\end{remark}

The morphism $g_{Y / T}$ commutes with \'etale morphisms.
More precisely:
\begin{proposition} \label{prop: Weil restriction for F commutes with etale}
    Assume that $\ff_{T}$ is a finite locally free morphism.
    Let $Y' \to Y$ be an \'etale morphism of $T$-schemes.
    Then the diagram
        \[
            \begin{CD}
                    \Weil_{\ff_{T}}(Y')
                @>>>
                    \Weil_{\ff_{T}}(Y)
                \\ @VV g_{Y' / T} V @VV g_{Y / T} V \\
                    Y' @>>> Y
            \end{CD}
        \]
    is cartesian.
\end{proposition}
\begin{proof}
  Let $Z$ be a $T$-scheme.
    The morphism $g_{Y / T}$ on $Z$-valued points
    is the map $\Hom_{T}(Z^{(p)}, Y) \to \Hom_{T}(Z, Y)$
    given by precomposition with
    the relative Frobenius $Z \to Z^{(p)}$ for $Z / T$.
      Therefore, what we want to show is the following:
    given a commutative diagram of $T$-schemes
        \[
            \begin{CD}
                Z @> \ff_{Z / T} >> Z^{(p)} \\
                @VVV @VVV \\
                Y' @>>> Y,
            \end{CD}
        \]
    there exists a unique morphism $Z^{(p)} \to Y'$
    that splits the diagram into two commutative triangles.
    The morphism $\ff_{Z / T}$ is a universal homeomorphism.
    The morphism $Y' \to Y$ is \'etale.
    Hence the existence of such a morphism follows from a simple application
    of the topological invariance of \'etale sites;
    see \cite[Corollary 4.4.1]{Bar10} for example.
\end{proof}

Here is how $W_{N}(R)$-algebras $W_{N}(Q)$ behave nicely
for relatively perfect $R$-algebras $Q$,
which will be the key point in the proof of Lemma \ref{l.Aone} below:

\begin{proposition} \label{p.wh}
	Let $R$ be an $\F_{p}$-algebra and $Q$ a relatively perfect $R$-algebra.
	Let $M \ge N \ge 1$ be integers.
	\begin{enumerate}
				\item   \label{i.wr}
		The  natural morphism 	$W_{M}(Q) \otimes_{W_{M}(R)} W_{N}(R)   \to  W_{N}(Q)$ is an  isomorphism.
		\item   \label{i.wt}
			$\Tor_{1}^{W_{M}(R)}(W_{M}(Q), W_{N}(R)) = 0$.
	\end{enumerate}
\end{proposition}

\begin{proof}
\eqref{i.wr}
For $M=N$ it is clear. 
By induction, it is enough to consider the case $M = N + 1$. Surjectivity follows since $W_{N+1}(Q)\to W_N(Q)$ is surjective. For injectivity, let $\vv^{N}$ be the $N$-fold Verschiebung $\G_{a} \hookrightarrow W_{N+1}$  and recall that $\vv^N(R)$ is an ideal in $W_N(R)$. 
Note that for any ring homomorphism $g\colon A\to B$ mapping an ideal $I $ into an ideal $J\subset B$ then $g$ induces an isomorphism $B\otimes_A A/I\simeq B/J$ if and only if $IB=J$. Hence, taking $A=W_{N + 1}(R)$ and $ B=W_{N + 1}(Q)$, it is enough to show that $ \vv^{N}(R) W_{N + 1}(Q)= \vv^{N}(Q)$	in $W_{N + 1}(Q)$. One inclusion is obvious. 
Let $(0, \dots, 0, y) \in \vv^{N}(Q)$.
Since $Q$ is relatively (semi-)perfect over $R$,
we can write $y$ as $\sum  x_{i}y_{i}^{p^{N}}$ for some $x_{i} \in R$ and $y_{i} \in Q$.
Then $(0, \dots, 0, y) = \sum  (0, \dots, 0, x_{i})(y_{i}, 0, \dots, 0)$,
which is in $\vv^{N}(R)W_{N + 1}(Q) $ \cite[Chapter 0, (1.1.9)]{Ill79}).

\eqref{i.wt}
Let $L = M - N$  and let $\ff $ also denote the Frobenius morphism on Witt vector rings.
By the exact sequence
	\[
				0
			\to
				W_{L}(R)
			\overset{\vv^{N}}{\to}
				W_{M}(R)
			\to
				W_{N}(R)
			\to
				0,\]
it is enough to show that the natural morphism
	$
			W_{M}(Q) \otimes_{W_{M}(R)} \vv^{N}(W_{L}(R))
		\to
			W_{M}(Q)
	$
is injective.
The $W_{M}(R)$-action on $\vv^{N}(W_{L}(R))$ factors through a $W_{L}(R)$-action. 
Since $x \vv(x') = \vv  (\ff(x) x')$ for $x, x' \in W(R)$ (\cite[Chapter 0, (1.3.6)]{Ill79}), the $W_{L}(R)$-module $\vv^{N}(W_{L}(R))$ is isomorphic to $W_{L}(R)$ with $W_{L}(R)$-action given by $\ff^{N}$. Therefore,
	\[
			W_{M}(Q) \otimes_{W_{M}(R)} \vv^{N}(W_{L}(R))
		\cong
				W_{L}(Q)
			\otimes_{W_{L}(R), \ff^{N}}
				W_{L}(R)
	\]
by the previous assertion.
The morphism $\ff^{N} \otimes \id_{W_{L}(R)}$ gives a bijection
from the right-hand side to $W_{L}(Q)$
by Equation (1) in the proof of \cite[Lemma\ 2]{Kat82}.
Therefore, we are reduced to showing that
$\vv^{N} \colon W_{L}(Q) \to W_{M}(Q)$ is injective.
But this is obvious.
\end{proof}


\section{Kato canonical lifting}
\label{s.cl}

Let $A$ be a ring, $I$ a nilpotent ideal of $A$ containing $p$ and $R=A/I$. If $R=k$ is a field of characteristic $p$ the examples to keep in mind are $A$ an artinian local ring with maximal ideal $I$, or $A=W_n(k)$ and $I=\vv( W_n(k))$.
Let $g\colon R\to Q$ be a ring homomorphism and assume that $Q$ is relatively perfect over $R$, namely, the ring homomorphism  
\[ 
	Q \otimes_{R, \ff_{R}} R \to Q, \
	x\otimes y \mapsto x^{p} g(y),
\] 
is bijective.  Recall that by \cite[Lemma 1.3]{Kat86} $Q$ is a formally \'etale ring over $R$.

Kato shows in \cite[Lemma 1]{Kat82} that there exists a {\it canonical lifting of $Q$ over $A$ (with respect to $I$)}, namely,  a formally \'etale ring $B$ over $A$ such that $B/IB\simeq Q$.
The canonical lifting is unique up to unique isomorphism.
We denote this $B$ by $h^{A}(Q)$.
In the same lemma, it is shown that
if $Q$ is flat over $R$, then
$B$ is also characterized as a unique flat $A$-algebra
with $B / I B \cong Q$.

The construction works as follows. Assume $I^n=0$;  let $r\geq n-1$ and $N> r$.
The $r$-th ghost component $\phi_r\colon W_N(A)\to A, (a_0,a_1,\dots, a_{N-1})\mapsto a_0^{p^r}+pa_1^{p^{r-1}}+\dots+p^ra_r$ factors through $W_N(R)$ thus endowing $A$ with a structure of $W_N(R)$-algebra; note that we have a commutative diagram
\[\xymatrix{W_N(A)\ar[d]\ar[r]^{\phi_r}&A\ar[d]\\
W_N(R)\ar[r]^{ \phi_r}&R
}\qquad
\xymatrix{(a_0,a_1,\dots)\ar[d]\ar[r]&a_0^{p^r}+pa_1^{p^{r-1}}+\dots+p^ra_r \ar[d]\\
	(\bar a_0,\bar a_1,\dots)\ar[r] &(\bar a_0)^{p^r}
}
\]
where $\bar a_i$ is the image of $a_i $   via the projection $A\to A/I=R$.

In \cite[Lemma 2]{Kat82} it is shown that
the relative perfectness of $R \to Q$ implies that
$W_{N}(R) \to W_{N}(Q)$ is formally \'etale,
and then it suffices to set $B=W_N(Q)\otimes_{W_N(R)}A$ to conclude.

\begin{remark}\label{rem.cl}
\begin{enumerate}
\item   If $A$ is a local artinian ring with maximal ideal $I$ and perfect residue field $R$ of characteristic $p$ then any relatively perfect ring $Q$ over $R$ is perfect and Kato canonical lifting $A\otimes_{W_N(R)}W_N(Q)$ is the lifting given by the Greenberg algebra associated with $A$ (see \cite[Corollary A.2]{Lip76}, \cite[Remark 2.2]{BGA18Per}).
\item \label{i.etale lifting} If $Q$ is \'etale over $R$, then it is relatively perfect \cite[Lemma 1.3]{Kat86} and the canonical lifting is the usual \'etale lifting of $Q$ over $A$. 
\item \label{i.basechange}
Let $J$ be an ideal contained in $I$ and let $A'=A/J$ so that $R=A'/IA'$. Then there is a natural isomorphism  between the canonical lifting $h^{A'}(Q)$ of $Q$ to $A'$ and $h^A(Q)\otimes_AA'$, both being $W_N(Q)\otimes_{W_N(R)} A'$ for $N$ large enough.
\end{enumerate}
\end{remark}

More generally, for any scheme $Y$ which is relatively perfect over $R$ there exists a {\it canonical lifting of $Y$ over $A$}, namely, a formally \'etale $A$-scheme $h^{A}(Y)$ whose base change to $R$ is $Y$.
Indeed one first defines the canonical lifting of $Y$ over $W_N(R)$ as the scheme $W_N(Y)=(|Y|, W_N(\cO_{Y}))$ whose underlying topological space is the space underlying $Y$ and whose structure sheaf is defined via $U \mapsto W_N(\cO_{Y}(U))$ on affine open subsets $U \subseteq Y$ (see \cite[Chapter 0, Section 1.5]{Ill79}). Then one takes the base change of $W_N(Y)$ along $\spec(A)\to \spec(W_N(R))$.

This construction is Zariski-local on $A$.
Therefore, we can generalize it even more as follows.

Let $S$ be a scheme. Recall that an ideal $\mathcal{I}\subseteq \mathcal{O}_S$ is called {\it locally nilpotent} if locally for the Zariski topology it is nilpotent, namely, $\mathcal{I}^n=0$, for some $n\in \mathbb{N}$ \cite[Chapter I, (4.5.16)]{GD71}.
\begin{proposition} \label{p.cl}
Let $S$ be a scheme. Let $T\hookrightarrow S$ be a closed immersion with locally nilpotent ideal of definition $\mathcal{I} \subseteq \mathcal{O}_{S}$ containing $p$.
Then for any scheme $Y$ that is relatively perfect over $T$, there exists a \textit{canonical lifting of $Y$ over $S$},
namely, a formally \'etale $S$-scheme $h^{S}(Y)$
whose base change to $T$ is $Y$:
    \[
        h^{S}(Y) \times_{S} T = Y.
    \]
Further, the formation of canonical liftings commutes with fiber products and if $\{Y_\lambda\}$ is an open covering of $Y$ then $\{h^{S}(Y_{\lambda})\}$ is an open covering of $h^{S}(Y)$.  \end{proposition}

If $S'$ is another closed subscheme of $S$ containing $T$ and $Y$ is a $T$-scheme,
then the morphism $T \hookrightarrow S'$ satisfies
the same conditions as $T \hookrightarrow S$ in Proposition \ref{p.cl}, since the ideal of $\mathcal{O}_{S'}$ defining $T$ is the image of the ideal of $\mathcal{O}_{S}$ defining $T$. Therefore,
by uniqueness of canonical liftings, there is a natural isomorphism (see Remark \ref{rem.cl} \eqref{i.basechange})
\begin{equation}\label{eq.cl-bc}
h^{S'}(Y) \simeq h^{S}(Y) \times_{S} S'.
\end{equation} 

We will also need an infinite level version (\cite[Definition 1]{Kat82}).
Let $A$ be a ring.
Let $I = I_{0} \supseteq I_{1} \supseteq \cdots$ be
a decreasing sequence of ideals of $A$.
Set $A_{n} = A / I_{n}$ and $R = A / I = A_{0}$.
Assume that $A \isomto \invlim_{n} A_{n}$, $p \in I$
and the ideal $I / I_{n} \subseteq A_{n}$ is nilpotent for all $n$.
Let $Q$ be a relatively perfect $R$-algebra.
For each integer $n \ge 0$,
we have the Kato canonical lifting $h^{A_{n}}(Q)$ of $Q$ over $A_{n}$ as above.
By Remark \ref{rem.cl} \eqref{i.basechange},
we have a natural surjective $A$-algebra homomorphism $h^{A_{n + 1}}(Q) \onto h^{A_{n}}(Q)$.
We define the \emph{Kato canonical lifting} of $Q$ over $A$ to be
    \begin{equation} \label{eq: Kato canonical lifting}
        h^{A}(Q) := \invlim_{n} h^{A_{n}}(Q).
    \end{equation}

Assume that $A$ is noetherian and $I_{n} = I^{n + 1}$ for all $n$.
For a flat relatively perfect $R$-algebra $Q$,
the $A_{n}$-algebra $h^{A_{n}}(Q)$ is flat and
the map $h^{A_{n + 1}}(Q) \onto h^{A_{n}}(Q)$ is surjective.
Therefore, by \cite[Tag 0912]{Sta20},
we know that $h^{A}(Q)$ is flat over $A$
and $h^{A}(Q) \tensor_{A} A_{n} \isomto h^{A_{n}}(Q)$ for all $n$.
(This is elementary if $A$ is a discrete valuation ring
and $I$ is its maximal ideal.)
Therefore, $h^{A}(Q)$ is characterized, in this case,
as a unique $I$-adically complete flat $A$-algebra with $h^{A}(Q) \tensor_{A} R \cong Q$.
Note that in the case $A$ is a complete discrete valuation ring with maximal ideal $I$ and $Q$ is a finite separable extension of the residue field $A/K$ then $h^{A}(Q)$ is the unique unramified extension of $A$ with residue field $Q$.


\section{Greenberg transform  of finite level: as a functor}
\label{s.gf}

Let $S$ be a scheme.
Let $T\hookrightarrow S$ be a closed immersion with locally nilpotent ideal of definition $\mathcal{I} \subseteq \mathcal{O}_{S}$ containing $p$
(which in particular implies that
$T$ is an $\F_{p}$-scheme).
For any $S$-scheme $X$, consider the functor
	\[
			F_{S / T, X}
		\colon
			(\rpsch / T)^{\mathrm{op}}
		\to
			\mathrm{Sets},
		\quad
			Y \mapsto X(h^{S}(Y)),
	\]
where
	$
			X(h^{S}(Y))
		:=
			\Hom_{S}(h^{S}(Y), X)
	$
and the superscript $\mathrm{op}$ denotes the opposite category.%
\footnote{The part ``$S / T$'' of the notation might look confusing since $T$ is an $S$-scheme.
But if $T \into S$ admits a left inverse $S \twoheadrightarrow T$ (the ``equal characteristic case''),
then $S / T$ literally makes sense and $F_{S / T, X}$ is the Weil restriction $\Weil_{S / T}(X)$.
We want to keep this intuition.}
If $F_{S / T, X}$ is representable, we call the representing object
the \textit{relatively perfect Greenberg transform of $X$ with respect to $T \hookrightarrow S$}
and denote it by $\gr^{\rp}(X)$ (or $\gr_{S}^{\rp}(X)$ or $\gr_{S / T}^{\rp}(X)$).
If this is the case, we have a canonical bijection
	\[
		\Hom_{T}(Y, \gr^\rp(X)) \simeq \Hom_{S}(h^{S}(Y), X)
	\]
functorial in relatively perfect $T$-schemes $Y$.
If $S'$ is another closed subscheme of $S$ containing $T$, then the closed immersion $h^{S'}(Y) \hookrightarrow h^{S}(Y)$  (see \eqref{eq.cl-bc}) gives a map
	\[
			\Hom_{S}(h^{S}(Y), X)
		\to
			\Hom_{S'}(h^{S'}(Y), X \times_{S} S'),
	\]
which then defines a canonical morphism
\[
\rho_{S / S'}^{\ast} \colon F_{S / T, X} \to F_{S' / T, X \times_{S} S'}
\]
of functors on $(\rpsch / T)^{\mathrm{op}}$.
If $S' = T$, then $F_{T / T, X \times_{S} T}$ is the restriction
of the representable functor $X \times_{S} T$ to $(\rpsch / T)^{\mathrm{op}}$,
which we denote by $(X \times_{S} T)^{\rp}$
(even if $T$ does not have finite locally free absolute Frobenius
and hence $(X \times_{S} T)^{\rp}$ is not necessarily representable).
Thus, we have a morphism of functors
	\begin{equation} \label{eq.rhoast}
			\rho_{S / T}^{\ast}
		\colon
			F_{S / T, X}
		\to
			(X \times_{S} T)^{\rp},
	\end{equation}
whose target is representable
if $\ff_{T}$ is finite locally free by Proposition \ref{p.rp}.

\begin{remark}\label{rem.F}
\begin{enumerate}
    \item \label{i.remFp}
	Note that the functor $F_{S / T, X}$ is the composition of functors $X\circ h^S$ where $X$ is identified with the associated functor from the category of (formally \'etale) $S$-schemes to $	\mathrm{Sets}$. Hence $X \mapsto F_{S / T, X}$  commutes with fiber products
    by \cite[Chapter 0, (1.2.2.1)]{GD71}.
    \item \label{i.remFr}
    If $Y$ is a relatively perfect $T$-scheme
    (no hypothesis on $T$
    other than $T$ being an $\F_{p}$-scheme),
    then $Y$  represents $F_{T/T,Y}$.
\end{enumerate}
\end{remark}

\begin{proposition} \label{p.sh}
	The functor $F_{S / T, X}$ satisfies the sheaf condition for the Zariski topology.
	That is, for any relatively perfect $T$-scheme $Y$ and any open covering $\{Y_{\lambda}\}$ of $Y$,
	the sequence
		\[
				F_{S / T, X}(Y)
			\to
				\prod_{\lambda} F_{S / T, X}(Y_{\lambda})
			\rightrightarrows
				\prod_{\lambda, \mu} F_{S / T, X}(Y_{\lambda} \cap Y_{\mu})
		\]
	is an equalizer sequence.
\end{proposition}

\begin{proof}
By Proposition \ref{p.cl}
we know that the canonical lifts $\{h^{S}(Y_{\lambda})\}$ form an open covering of $h^{S}(Y)$
and $h^{S}(Y_{\lambda}) \cap h^{S}(Y_{\mu}) = h^{S}(Y_{\lambda} \cap Y_{\mu})$. Then the assertion of the proposition follows since $X$ is a sheaf for the Zariski topology.
\end{proof}

\begin{proposition} \label{r.cartesian}
	If $X'\to X$ is a formally \'etale morphism of $S$-schemes,
	then the diagram
		\[
			\xymatrix{
				F_{S / T, X'} \ar[d]^{\rho_{S/T}^\ast}\ar[rr]& & F_{S / T, X}\ar[d]^{\rho_{S/T}^\ast} \\
				(X' \times_{S} T)^{\rp} \ar[rr]& &(X \times_{S} T)^{\rp},
			}
		\]
	is cartesian, where $\rho_{S/T}^\ast$ was defined in \eqref{eq.rhoast}.
\end{proposition}

\begin{proof}
	We need to show that for any relatively perfect $T$-scheme $Y$,
	the square becomes a cartesian diagram of sets on taking $Y$-valued points.
	In other words, given a commutative square
		\[
			\begin{CD}
					Y
				@> \text{incl} >>
					h^{S}(Y)
				\\
				@VVV @VVV
				\\
					X'
				@>>>
					X,
			\end{CD}
		\]
	we want show that there exists a unique morphism $h^{S}(Y) \to X'$
	that splits the square into two commutative triangles.
	Since the closed immersion $T \hookrightarrow S$ is defined by
	a locally nilpotent ideal sheaf by assumption, so is its base change $Y \hookrightarrow h^{S}(Y)$ by the morphism $h^{S}(Y) \to S$.
Therefore the result follows from the formal \'etaleness of $X' \to X$ and \cite[Remark (17.1.2) (iv)]{Gro67},
\end{proof}

For a morphism $F' \to F$ between functors
$(\rpsch / T)^{\mathrm{op}} \to \mathrm{Sets}$,
we say that it is \emph{relatively representable by open immersions}
if for any relatively perfect $T$-scheme $Y$ and
any section $y\in F(Y)$
the induced morphism of functors $F' \times_{F} Y \to Y$ is
an open immersions of $T$-schemes,
where the morphism $Y \to F$ in the fiber product is
the one associated with $y$.
See \cite[Chapter 0, (1.1.3), (1.1.4)]{GD71}
for the background material on the functorial language.

\begin{proposition} \label{p.px}
For any open immersion of $S$-schemes $X' \hookrightarrow X$ the induced morphism $F_{S / T, X'} \to F_{S / T, X}$ is relatively represented by open immersions.
Moreover, for any open covering $\{X_{\lambda}\}$ of $X$, we have $F_{S / T, X} = \bigcup_{\lambda} F_{S / T, X_{\lambda}}$ as Zariski sheaves on $(\rpsch / T)$.
\end{proposition}

\begin{proof}
	This follows from Proposition \ref{r.cartesian}.
\end{proof}

For an open immersion $j \colon T' \hookrightarrow T$ and
a functor $F' \colon (\rpsch / T')^{\mathrm{op}} \to \mathrm{Sets}$,
we define $j_{!} F'$ to be the functor $(\rpsch / T)^{\mathrm{op}} \to \mathrm{Sets}$
that assigns to each $T$-scheme $Y$
the set $F'(Y)$ if the morphism $Y \to T$ factors through $T'$
and $\emptyset$ otherwise.
If $F'$ is represented by some $T'$-scheme $Y'$,
then $j_{!} F'$ is represented by $Y'$ viewed as a $T$-scheme.

\begin{proposition} \label{p.pt}
	For any open immersion $j \colon T' \hookrightarrow T$,
	the restriction $j^{\ast} F_{S / T, X}$ of $F_{S / T, X}$ to $(\rpsch / T')$ is 	naturally isomorphic to $F_{S' / T', X'}$
	(where $S' \subseteq S$ is the open subscheme corresponding to $T'$ and	$X' = X \times_{S} S'$).
	The resulting morphism $j_{!} F_{S' / T', X'} \to F_{S / T, X}$ is represented by open immersions.
	For any open covering $\{j_{\lambda} \colon T_{\lambda} \hookrightarrow T\}$ of $T$,
	we have
		\[
				F_{S / T, X}
			=
				\bigcup_{\lambda}
					(j_{\lambda})_{!}
					F_{S_{\lambda} / T_{\lambda}, X_{\lambda}}
		\]
	as Zariski sheaves on $(\rpsch / T)$
	(where $S_{\lambda}$ and $X_{\lambda}$ are similarly defined).
\end{proposition}

\begin{proof}
	For any relatively perfect $T'$-scheme $Y'$, we have $h^{S'}(Y') \cong h^{S}(Y')$ since $h^{S'}(Y')$ is formally \'etale over $S$ and it lifts the $T$-scheme $Y'$. 
	It follows that
		\[
				j^{\ast} F_{S / T, X}(Y')
			=
				F_{S / T, X}(Y')
			=
				X(h^{S}(Y'))
			=
				X(h^{S'}(Y'))
			=
				X'(h^{S'}(Y'))
			=
				F_{S' / T', X'}(Y'),
		\]
	thus $j^{\ast} F_{S / T, X} \cong F_{S' / T', X'}$.
	Further, we have a cartesian square
		\[
			\begin{CD}
				j_{!} F_{S' / T', X'} @>>> F_{S / T, X} \\
				@VVV @VVV \\
				T' @>>> T
			\end{CD}
		\]
	as functors $(\rpsch / T)^{\mathrm{op}} \to (\mathrm{Sets})$.
	The rest follows from this.
\end{proof}


\section{Greenberg transform  of finite level: representability}\label{s.existence}
We now give a class of examples where the functor $F_{S / T, X}$ is representable, namely, the relatively perfect Greenberg transform $\gr_{S / T}^{\rp}(X)$ exists.  

Let $\mathcal{I} = \mathcal{I}_{0} \supseteq \mathcal{I}_{1} \supseteq \dots
\supseteq \mathcal{I}_{n} = 0$
be a decreasing sequence of quasi-coherent sheaves of ideals on $S$.
Let $S_{i} \subseteq S$ be the closed subscheme defined by $\mathcal{I}_{i}$,
so $S_{n} = S$ and $S_{0} = T$.
Assume the following:
\begin{enumerate}
\item \label{i.pr}
 The product ideal $\mathcal{I} \cdot \mathcal{I}_{i}$ is contained in $\mathcal{I}_{i+1}$. In particular $\mathcal{I}\mathcal{O}_{S_i}$ is nilpotent.
\item \label{i.fr}
The absolute Frobenius morphism $\ff_{T}$ of $T$ is
finite locally free.
\item \label{i.id}
The sheaf $\mathcal{I}_{i} / \mathcal{I}_{i + 1}$
is a finite locally free $\mathcal{O}_{T}$-module, where the $\mathcal{O}_{T}$-action is induced by the $\cO_{S_{i+1}}$-module structure on $\mathcal{I}_{i} / \mathcal{I}_{i + 1}$ via \eqref{i.pr}.
\end{enumerate}
Here are two examples.
Let $k$ be a field of characteristic $p$ such that $[k : k^{p}] < \infty$.
Let $A$ be an artinian local ring with maximal ideal $I$ and residue field $k$.
Then $S = \spec A$, $T = \spec k$, $\mathcal{I}_{i} = I^{i + 1}$ satisfy these conditions.
Also, let $k$ be as above.
Let $A = W_{n}(k)$ and $I_{i} = \vv^{i + 1} W_{n - i - 1}(k)$,
where $\vv$ is the Verschiebung.
Then $S = \spec A$, $T = \spec k$, $\mathcal{I}_{i} = I_{i}$ satisfy these conditions.

For an $S$-scheme $X$, set $X_{i} = X \times_{S} S_{i}$.
Denote $F_{i} = F_{i, X} = F_{S_{i} / T, X_i}$ and,
if it is representable, denote its representing object by
$\gr_{S_{i} / T}^{\rp}(X_{i}) = \gr_{i}^{\rp}(X)$.
We have a canonical morphism 
        $
			\rho_{i, j}^{\ast}
		=
			\rho_{S_{i} / S_{j}}^{\ast}
		\colon
			F_{i, X}
		\to
			F_{j, X}
	$
over $T$ for any $i \ge j$.

For proving the existence of $\gr_i^\rp(X)$ in this situation,
we imitate the proof of the representability of the Weil restriction functor in \cite[Section 7.6]{BLR90} and the proof of the existence of classical Greenberg transform in \cite{Gre61}.
We first consider the following results:

\begin{lemma} \label{l.Aone}
	Assume that $S$ is affine and that Conditions \eqref{i.fr} and \eqref{i.id} above hold with ``finite locally free'' replaced by ``finite free''.
	Then $\gr_{i}^{\rp}(\A_{S}^{1})$ exists.
	It is isomorphic to the relative perfection of an affine space over $T$.
\end{lemma}

\begin{proof}
	Write $S = \spec A$, $S_{j} = \spec A / I_{j}$, $j \geq 0$, and
	$T = \spec R = \spec A / I$.
	Set $I_{-1} = A$.
	By Proposition \ref{p.sh}, it is enough to test functors $F_j=F_{S_j/T,\A_{S_j}^{1}}$ by relatively perfect affine $T$-schemes.
	Let $Q$ be a relatively perfect $R$-algebra.
	Let $r \ge n$ and $N > r$.
	Then $F_{j}(Q) = W_{N}(Q) \otimes_{W_{N}(R)} A / I_{j}$ is Kato's canonical lifting of $Q$ as in Section \ref{s.cl}.
	 
	For $i \ge j \ge 0$, consider the exact sequence
		\begin{equation} \label{eq.dv}
				W_{N}(Q) \otimes_{W_{N}(R)} I_{j} / I_{i}
			\to
				W_{N}(Q) \otimes_{W_{N}(R)} I_{j - 1} / I_{i}
			\to
				W_{N}(Q) \otimes_{W_{N}(R)} I_{j - 1} / I_{j}
			\to
				0.
		\end{equation}
	We show that: the first morphism is injective;
	the third term is a finite free $Q$-module (via $W_{N}(Q) \twoheadrightarrow Q$)
	admitting a basis that does not depend on $Q$;
	and the second morphism admits a set-theoretic section
	(not necessarily respecting the group structures) functorial in $Q$.
	Recall that we are endowing $A$ with the action of $W_{N}(R)$
	via the $r$-th ghost component map $\phi_{r}$.
	Hence $I_{j - 1} / I_{j}$ here is endowed with the $W_{N}(R)$-action
	via the quotient $W_{N}(R) \twoheadrightarrow R$ and $\ff_R^{r} \colon R \to R$. 
	By Conditions \eqref{i.fr} and \eqref{i.id},
	this $I_{j - 1} / I_{j}$ is a finite free $R$-module.
	Hence $\Tor_{1}^{W_{N}(R)}(W_{N}(Q), I_{j - 1} / I_{j}) = 0$
	by Proposition \ref{p.wh} \eqref{i.wt}.
	Therefore the first morphism in \eqref{eq.dv} is injective.
	The third term of \eqref{eq.dv} is naturally isomorphic to
		\[
				(W_{N}(Q) \otimes_{W_{N}(R)} R)
			\otimes_{R, \ff^{r}}
				I_{j - 1} / I_{j},
		\]
	where the $W_{N}(R)$-action on $R$ is via the quotient map $W_{N}(R) \twoheadrightarrow R$.
	Hence, this is naturally isomorphic to $Q \otimes_{R, \ff^{r}} I_{j - 1} / I_{j}$
	by Proposition \ref{p.wh} \eqref{i.wr}.
	Let $e_{1}, \dots, e_{m}$ be a basis of the finite free $R$-module $I_{j - 1} / I_{j}$
	with the $R$-action given by $\ff^{r}$.
	Then $Q \otimes_{R, \ff^{r}} I_{j - 1} / I_{j}$ is a finite free $Q$-module
	with basis $\{1 \otimes e_{l}\}_{l}$.
	A section to the second morphism in \eqref{eq.dv} is given by the map
		\begin{align*}
					Q \otimes_{R, \ff^{r}} I_{j - 1} / I_{j}
				=
					\bigoplus_{l = 1}^{m} Q \otimes e_{l}
			&	\to
					W_{N}(Q) \otimes_{W_{N}(R)} I_{j - 1} / I_{i},
			\\
					\sum_{l = 1}^{m} y_{l} \otimes e_{l}
			&	\mapsto
					\sum_{l = 1}^{m} (y_{l}, 0, \dots, 0) \otimes \Tilde{e}_{l},
		\end{align*}
	where $\Tilde{e}_{l} \in I_{j - 1} / I_{i}$ is a lift of $e_{l} \in I_{j - 1} / I_{j}$.
	This is functorial in $Q$.
	
	Now, let $F_{i, j}(Q) = W_{N}(Q) \otimes_{W_{N}(R)} I_{j} / I_{i}$,
	so that $F_{i} = F_{i, -1}$.
	What we saw means that	
	$F_{i, j - 1} \cong F_{i, j} \times F_{j, j - 1}$
	as functors $(\rpaff / R)^{\mathrm{op}} \to (\mathrm{Sets})$
	and $F_{j, j - 1}$ is isomorphic to the $m$-fold product $F_0\times \cdots \times F_0$
	and thus is represented by the relative perfection of an affine space over $R$, $ (\A_R^m)^\rp$.
	Hence $F_{i} \cong F_{i, i - 1} \times \dots
	\times F_{1, 0} \times F_{0, -1}$
	is represented by the relative perfection of an affine space over $R$.
\end{proof}

\begin{lemma} \label{l.GrAff}
	Under the same assumptions as the previous lemma,
	$\gr_{i}^{\rp}(X)$ exists for any affine $S$-scheme $X$.
\end{lemma}

\begin{proof}
	Write $X = \spec B$ and $S = \spec A$.
	Write $B$ as a quotient of a (possibly infinitely many variable) polynomial ring
	$B = A[x_{\lambda} \,|\, \lambda \in \Lambda] / (f_{\mu} \,|\, \mu \in M)$.
	Then $B$ is the coequalizer of two $A$-algebra homomorphisms
		\[
				A[y_{\mu} \,|\, \mu \in M]
			\rightrightarrows
				A[x_{\lambda} \,|\, \lambda \in \Lambda]
			\to
				B,
		\]
	where the upper homomorphism maps $y_{\mu}$ to $0$ and
	the lower maps $y_{\mu}$ to $f_{\mu}$;
	hence $X$ can be viewed as an equalizer
	 	\[
	 			X
	 		\to
	 			\A_{S}^{\Lambda}
	 		\rightrightarrows
	 			\A_{S}^{M}
	 	\]
	in the category of $S$-schemes.
	Let $\alpha$ (respectively, $\beta$) be the upper (respectively, lower) morphism.
	By the previous lemma and Remark \ref{rem.F} \eqref{i.remFp},
	$\gr_{i}^{\rp}(\A_{S}^{\Lambda}) = \gr_{i}^{\rp}(\A_{S})^{\Lambda}$ exists and is affine.
	Then $F_{i ,X}$ is representable by the equalizer of the two morphisms of affine schemes,
	 	\[
	 			\gr_{i}^{\rp}(\A_{S}^{\Lambda})
	 		\rightrightarrows
	 			\gr_{i}^{\rp}(\A_{S}^{M}),
	 	\]
	where the upper (respectively, lower) morphism is $\gr_{i}^{\rp}(\alpha)$ (respectively, $\gr_{i}^{\rp}( \beta)$).
	Note that the above equalizer is relatively perfect over $T$ by \cite[Lemma 1.2 (iii)]{Kat86}.
\end{proof}

We can now prove the representability of $F_{n,X}$ in general.
\begin{proposition}\label{p.gr}
	Let $S, T$ satisfy Conditions \eqref{i.pr}, \eqref{i.fr}, \eqref{i.id}
	and let $X$ be any $S$-scheme.
	The functor $F_{j, X}$ is representable for any $j$.
	The morphisms $\rho_{i, j}^{\ast} \colon F_{i, X} \to F_{j, X}$ are affine for any $i \ge j$.
\end{proposition}

\begin{proof}
Representability follows from the previous lemma and patching by Propositions\ \ref{p.px}, \ref{p.pt}.
Affineness follows by the cartesian square in Proposition \ref{r.cartesian}.
\end{proof}

We will now see that the structure of $\gr_{n}^{\rp}(X)$ is relatively simple if $X / S$ is smooth.
First, in view of Proposition \ref{r.cartesian},
it makes sense to talk about properties of $\gr_{n}^{\rp}(X)$ ``Zariski (or \'etale) local on $X$ and $S$''.

\begin{proposition} \label{p.SmSch}
	Assume that $X$ is smooth over $S$.
	Then for any $i$, the morphism $\gr_{i + 1}^{\rp}(X) \to \gr_{i}^{\rp}(X)$ is isomorphic,
	Zariski locally on $X$ and $S$, to the second projection
	$(\A_{T}^{m})^{\rp} \times_{T} \gr_{i}^{\rp}(X) \twoheadrightarrow \gr_{i}^{\rp}(X)$
	for some $m$.
	In particular, it is faithfully flat, and admits a section Zariski locally on $X$ and $S$.
\end{proposition}

\begin{proof}
	We may assume that $X$ is \'etale over some affine space $\A_{S}^{l}$ over $S$.
	The statement is true for $\A_{S}^{1}$ and hence for $\A_{S}^{l}$ by the proof of Lemma \ref{l.Aone}.
	The diagram
		\[
			\begin{CD}
					\gr_{i + 1}^{\rp}(X)
				@>>>
					\gr_{i}^{\rp}(X)
				\\
				@VVV @VVV
				\\
					\gr_{i + 1}^{\rp}(\A_{S}^{l})
				@>>>
					\gr_{i}^{\rp}(\A_{S}^{l})
			\end{CD}
		\]
	is cartesian by Proposition \ref{r.cartesian}.
	Hence the statement is true for $X$.
\end{proof}

\begin{proposition} \label{p.SmGp}
	Assume that $X = G$ is a smooth group scheme over $S$.
	Then the kernel of $\gr_{i + 1}^{\rp}(X) \twoheadrightarrow \gr_{i}^{\rp}(X)$ is
	canonically isomorphic to
	$\bigl( \mathrm{Lie}(G \times_{S} T) \otimes_{\cO_{T}} \mathcal{I}_{i} / \mathcal{I}_{i + 1} \bigr)^{\rp}$.
\end{proposition}

\begin{proof}
	First assume that $S = \spec A$ (and hence $T = \spec R$ also) are affine.
	Let $Q$ be a relatively perfect $R$-algebra.
	Then the morphism $\gr_{i + 1}^{\rp}(X) \twoheadrightarrow \gr_{i}^{\rp}(X)$ on $Q$-valued points is
	$G(W_{N}(Q) \otimes_{W_{N}(R)} A / I_{i + 1}) \twoheadrightarrow G(W_{N}(Q) \otimes_{W_{N}(R)} A / I_{i})$.
	The kernel of $W_{N}(Q) \otimes_{W_{N}(R)} A / I_{i + 1} \twoheadrightarrow W_{N}(Q) \otimes_{W_{N}(R)} A / I_{i}$
	is $Q \otimes_{R, \ff_{R}^{r}} I_{i} / I_{i + 1}$
	by Proposition \ref{p.wh}, which itself is isomorphic to $Q \otimes_{R} I_{i} / I_{i + 1}$
	by the relative perfectness of $Q$.
	Therefore the kernel in question on $Q$-valued points is
	$\Hom_{A-\mathrm{module}}(\omega_{G / S}, Q \otimes_{R} I_{i} / I_{i + 1})$
	by the theorem of infinitesimal points \cite[Chapter II, Section 4, No.\ 3.5]{DG70},
	where $\omega_{G / S}$ is the module of K\"ahler differentials
	pulled back to $S$ by the identity section $S \into G$.
	This $A$-module is isomorphic to
	$\mathrm{Lie}(G \times_{S} T) \otimes_{R} Q \otimes_{R} I_{i} / I_{i + 1}$
	by the smoothness of $G / S$.
	This proves the statement for the case $S$ is affine.
	The general case follows from this by patching.
\end{proof}

\begin{remark}
    Propositions \ref{p.SmSch} and \ref{p.SmGp} can be unified with stronger conclusions as follows.
    (This is a generalization of the Greenberg structure theorem \cite[Corollary A.13]{BGA18Gre}.)
    Assume that $X$ is smooth over $S$.
    Let $\Theta_{X_{T} / T}$ be the tangent sheaf of $X_{T} = X \times_{S} T$ over $T$,
    which is the dual of $\Omega_{X_{T} / T}^{1}$.
    Consider the vector bundle
    $\Theta_{X_{T} / T} \otimes_{\Order_{T}} \mathcal{I}_{i} / \mathcal{I}_{i + 1}$
    on $X_{T}$.
    Its relative perfection
    $(\Theta_{X_{T} / T} \otimes_{\Order_{T}} \mathcal{I}_{i} / \mathcal{I}_{i + 1})^{\RP}$
    over $T$ gives a group scheme over $X_{T}^{\RP}$.
    Consider its base change
    $(\Theta_{X_{T} / T} \otimes_{\Order_{T}} \mathcal{I}_{i} / \mathcal{I}_{i + 1})_{\Gr_{i}^{\RP}(X)}^{\RP}$
    to $\Gr_{i}^{\RP}(X)$ by $\Gr_{i}^{\RP}(X) \onto X_{T}^{\RP}$.
    Then it can be shown that the morphism
    $\Gr_{i + 1}^{\RP}(X) \onto \Gr_{i}^{\RP}(X)$ has a canonical structure of a  Zariski torsor under the group scheme
    $(\Theta_{X_{T} / T} \otimes_{\Order_{T}} \mathcal{I}_{i} / \mathcal{I}_{i + 1})_{\Gr_{i}^{\RP}(X)}^{\RP}$
    over $\Gr_{i}^{\RP}(X)$.%
    \footnote{In fact, this torsor is trivial over any affine open
    of $\Gr_{i}^{\RP}(X)$ since the group
    $\Theta_{X_{T} / T} \otimes_{\Order_{T}} \mathcal{I}_{i} / \mathcal{I}_{i + 1}$
    is a vector bundle and
    an fpqc torsor under a quasi-coherent sheaf over an affine scheme is
    trivial by fpqc descent.}
    This implies Proposition \ref{p.SmSch}.
    If $X = G$ is a smooth group scheme,
    then this group scheme action on the identity section defines an isomorphism stated in Proposition \ref{p.SmGp}.
    The proof is longer and more complicated than the proofs of these propositions.
    Also we do not use this fact.
    Therefore we do not prove it.
\end{remark}


\section{Greenberg transform of infinite level}
\label{sec: Greenberg transform of infinite level}
We now restrict to the affine base case, in order to apply an algebraization result in \cite{Bha16}.
Let $A$ be a commutative ring.  
Let $I = I_{0} \supseteq I_{1} \supseteq \cdots$ be
a decreasing sequence of ideals of $A$
such that $A \overset{\sim}{\to} \varprojlim_{n} A / I_{n}$ and $p \in I$.
Let $A_{n} = A / I_{n}$ and set $R = A / I = A_{0}$.
Assume that conditions analogous to Conditions \eqref{i.pr}, \eqref{i.fr}, \eqref{i.id} of Section \ref{s.existence} hold, namely:
\begin{itemize}
	\item
		The product ideal $I \cdot I_{i}$ is contained in $I_{i + 1}$. In particular, $I/I_n$ is nilpotent.
	\item
		The absolute Frobenius morphism $\ff_{R}$ of $R$ is
            finite locally free.
	\item
		The $R$-module $I_{i} / I_{i + 1}$ is finite locally free.
\end{itemize}
For any relatively perfect $R$-algebra $Q$,
we have the Kato canonical lifting
$h^{A}(Q) = \varprojlim_{n} h^{A_{n}}(Q)$
of $Q$ over $A$,
as defined in \eqref{eq: Kato canonical lifting}.
For any $A$-scheme $X$, consider the functor
	\[
			F_{X} \colon (\rpaff / R)^{\mathrm{op}} \to (\mathrm{Sets}),
		\quad
			Q \mapsto X(h^{A}(Q)).
	\]

\begin{proposition}\label{p.adjinf}
	The natural morphisms $F_{X} \to F_{X \times_{A} A_{n}}$ induces an isomorphism
	$F_{X} \overset{\sim}{\to} \varprojlim_{n} F_{X \times_{A} A_{n}}$
	of functors $(\rpaff / R)^{\mathrm{op}} \to (\mathrm{Sets})$.
	The inverse limit $\varprojlim_{n} \gr_{n}^{\rp}(X)$ is a relatively perfect $R$-scheme representing $F_{X}$.
	That is, for any relatively perfect $R$-algebra $Q$, the natural morphism
		\[
				X(h^{A}(Q))
			\to
				\Hom_{R} \bigl(
					\spec Q, \varprojlim_{n} \gr_{n}^{\rp}(X)
				\bigr)
		\]
	is a bijection functorial in $Q$.
\end{proposition}

\begin{proof}
	We have
		\[
				X(h^{A}(Q))
			\overset{\sim}{\to}
				\varprojlim_{n} X(h^{A_{n}}(Q))
		\]
	by Bhatt-Gabber's algebraization theorem
	\cite[Theorem 4.1, Remark 4.3, Remark 4.6]{Bha16}.
	Hence $F_{X} \overset{\sim}{\to} \varprojlim_{n} F_{X \times_{A} A_{n}}$.
	The inverse limit $\varprojlim_{n} \gr_{n}^{\rp}(X)$ is a scheme
	since the transition morphisms are affine by Proposition \ref{p.gr}.
	The rest follows.
\end{proof}
This proves Theorem \ref{thm: Greenberg transform} as a special case.
We will write $\gr^\rp(X)$ for $F_X$ in the following; in particular \begin{equation}\label{eq.adj}
	\gr^\rp(X)(R)=X(A). \end{equation}

\begin{remark}
    For affine $X$, the Bhatt-Gabber theorem used in the proof of Proposition \ref{p.adjinf} is trivial.
    The real need of this theorem is for non-affine $X$ (for example, N\'eron models).
    The same already happens in the case where $R$ is a perfect field.
    In this case, Proposition \ref{p.adjinf} corresponds to \cite[Proposition 14.2]{BGA18Gre}.
    The reduction to the affine case in the proof of \cite[Proposition 14.2]{BGA18Gre} is not correct since \cite[Equation (4)]{BGA18Gre} does not extend to the infinite level case.
    Instead, the Bhatt-Gabber theorem is again needed in the proof of \cite[Proposition 14.2]{BGA18Gre}.
\end{remark}

We describe the structure of $\Gr^{\RP}(X)$ when $X$ is smooth over $A$.
Denote $S = \Spec A$ and $T = \Spec R$.
For a formally \'etale morphism $X' \to X$ over $A$, we have a natural cartesian diagram
    \[
        \begin{CD}
                \Gr^{\RP}(X')
            @>>>
                \Gr^{\RP}(X)
            \\
            @VVV @VVV
            \\
                (X' \times_{S} T)^{\RP}
            @>>>
                (X \times_{S} T)^{\RP}
        \end{CD}
    \]
by taking an inverse limit in the diagram in Proposition \ref{r.cartesian}.
Therefore, it makes sense to talk about properties of $\Gr^{\RP}(X)$
Zariski local on $X$
(but not in $S$ or $T$ since $h$ does not preserve localization
in this infinite level situation).

\begin{proposition} \label{prop: structure of infinite level Gr for smooth X}
    Assume that $A$ is local and $X$ is smooth over $A$.
    Let $n \ge 0$ be an integer.
    Then the morphism $\Gr^{\RP}(X) \to  \Gr_{n}^{\RP}(X)$ is isomorphic,
    Zariski locally on $X$, to the second projection
        \[
                (\Affine_{T}^{\N})^{\RP} \times_{T} \Gr_{n}^{\RP}(X)
            \to
                \Gr_{n}^{\RP}(X).
        \]
    In particular, it is faithfully flat, and admits a section Zariski locally on $X$.
\end{proposition}

\begin{proof}
    We may assume that $X$ is \'etale over some affine space $\Affine_{S}^{l}$ over $S$.
    Then for any $i \ge n$,
    the morphism $\Gr_{i + 1}^{\RP}(X) \onto \Gr_{i}^{\RP}(X)$ is isomorphic to the second projection
    $\Affine_{T}^{m_{i}} \times_{T} \Gr_{i}^{\RP}(X) \onto \Gr_{i}^{\RP}(X)$
    for some integer $m_{i}$
    by the proof of Proposition \ref{p.SmSch}.
    Hence the morphism $\Gr_{i + 1}^{\RP}(X) \onto \Gr_{n}^{\RP}(X)$ is isomorphic to the second projection
    $\Affine_{T}^{m_{i}'} \times_{T} \Gr_{n}^{\RP}(X) \onto \Gr_{n}^{\RP}(X)$,
    where $m_{i}' = m_{n} + \dots + m_{i}$.
    Taking the inverse limit in $i$, we get the result.
\end{proof}


\section{Greenberg transform before relative perfection}
\label{sec: Greenberg transform before relative perfection}
In this section, we will define a (twisted) Greenberg transform
before relative perfection in some cases.
Its relative perfection recovers the relatively perfect Greenberg transform.

The caveat is that this construction depends on the choice of an integer $r$ such that $I^{r + 1} = 0$ (where $R = A / I$).
Therefore, there is a difficulty in defining a Greenberg transform of infinite level, which is of course very important since the main objects we want apply the Greenberg transform to are over complete discrete valuation rings.
For complete discrete valuation rings
(and more general complete noetherian local rings),
we can still take a certain limit of the Greenberg transforms of finite levels.
But this limit is ``wrong'': the resulting object is automatically relatively perfect
and hence does not recover the usual Greenberg transform in the case of perfect residue fields.
In this sense, the relatively perfect Greenberg transform still seems to be the only possible construction
in the infinite level, imperfect residue field situation.

Also, we cannot hope for the same level of generality as the relatively perfect Greenberg transform:
the ring $A$ is taken to be an artinian local ring.
It is not clear how to generalize this to the nilpotent-thickening situation $T \into S$
that we did for the relatively perfect Greenberg transform.
If we naively generalize the Greenberg transform below for the situation of $T \into S$
with $T$ and $S$ affine
satisfying the conditions \eqref{i.pr}, \eqref{i.fr}, \eqref{i.id}
in Section \ref{s.existence},
the resulting Greenberg transform is not representable in general.

Let $k$ be a field of characteristic $p > 0$ with $[k : k^{p}] < \infty$.
Let $A$ be an artinian local ring with maximal ideal $I$ and residue field $k$.
Fix an integer $r \ge 0$ such that $I^{r + 1} = 0$.
As before, for $N > r$,
we give $A$ a canonical $W_{N}(k)$-algebra structure by the lifted $r$-th ghost component map.
In this section, we mainly work with the fppf topology.
First, we define a twisted version of Greenberg algebra.

\begin{definition}
	Let $N > r$ be an integer.
	Define an fppf sheaf of $A$-algebras $h_{r}^{A}$ over $k$ to be the fppf sheafification of the presheaf
		\[
				Q
			\mapsto
				W_{N}(Q) \otimes_{W_{N}(k)} A,
		\]
	where $Q$ runs through $k$-algebras.
	In other words, we define $h_{r}^{A} = W_{N} \otimes_{W_{N}(k)} A$,
	where $W_{N}$ is the affine scheme of Witt vectors over $k$ viewed as an fppf sheaf and $A$ is viewed as a constant sheaf.
	We call $h_{r}^{A}$ the \emph{$r$-twisted Greenberg algebra} over $k$
	associated with $A$.
\end{definition}

\begin{proposition} \label{prop: Greenberg algebra indep of N}
	The sheaf $h_{r}^{A}$ is independent of the choice of $N$.
	More precisely, the natural morphism
	$W_{M} \otimes_{W_{M}(k)} A \to W_{N} \otimes_{W_{N}(k)} A$
	is an isomorphism for any $M \ge N$.
\end{proposition}

\begin{proof}
We have $p^{N} = 0$ in $W_{N}$.
	Also, $p^{N} W_{M} = \vv^{N} \ff^{N} W_{M} = \vv^{N} W_{M}$	since $\ff \colon W_{M} \to W_{M}$ is an epimorphism (in the fppf topology).
	Thus $W_{M} / p^{N} W_{M} \cong W_{N}$ (as fppf sheaves).
	Hence,
	    \[
	               W_{M} \otimes_{W_{M}(k)} A
	           \cong
	               W_{N} \otimes_{W_{M}(k)} A
	           \cong
	               W_{N} \otimes_{W_{N}(k)} A.
	    \]
\end{proof}

When $k$ is perfect, we also have the usual Greenberg algebra
$h^{A} = W_{N} \otimes_{W_{N}(k), \mathrm{can}} A$,
where the $W_{N}(k)$-algebra structure on $A$ is given by
the usual (that is, Teichm\"uller) one (\cite[Section 3]{BGA18Gre}).
The natural morphism
$\ff^{r} \otimes \id \colon h_{r}^{A} \to h^{A}$
is not an isomorphism unless $r = 0$ and $A = k$.

Now we define an $r$-twisted Greenberg transform as a functor:

\begin{definition}
	Let $X$ be an $A$-scheme.
	Define a functor $F_{X}^{r}$ on the category of $k$-algebras to the category of sets
	by the formula
		\[
			F_{X}^{r}(Q) = X(h_{r}^{A}(Q)) \quad (= \Hom_{A}(\Spec h_{r}^{A}(Q), X)).
		\]
\end{definition}

We will show that this functor $F_{X}^{r}$ is representable by a $k$-scheme
(and call the representing object the \emph{$r$-twisted Greenberg transform} of $X$).
In the rest of this section, we set $N = r + 1$,
and fix a $p$-basis
$\mathcal{B} = \{t(1), \dots, t(d)\}$ of $k$
and its lift $\{\Tilde{t}(1), \dots, \Tilde{t}(d)\}$ to $A$.
We recall the Cohen ring scheme defined by Schoeller:

\begin{definition}[{\cite[Section 2.1]{Sch72}}]
 Let $n \ge 0$ be an integer.
	Let $Q$ be a $k$-algebra.
	We define $C_{n + 1}(Q)$ (denoted by $\mathcal{J}_{n + 1}^{\mathcal{B}}(Q)$ in \cite{Sch72}) to be the $W_{n + 1}(Q)$-subalgebra of $W_{n + 1}(Q \otimes_{k, \Fr_{k}^{n}} k)$ generated by the elements $(1 \otimes t(i), 0, \dots, 0)$ over all integers $1 \le i \le d$.
	This is naturally a functor in $Q$.
\end{definition}

For clarity and readability of future proofs, we have chosen to avoid indices and multi-index notation used in \cite{Sch72}.

\begin{proposition}[{\cite[Proposition 3.1]{Sch72}}]
    \label{prop: Cohen ring is artinian with max ideal p}
	The ring $C_{n + 1}(k)$ is an artinian local ring with maximal ideal $(p)$	and residue field $k$.
\end{proposition}

This is why $C_{n + 1}$ is better behaved than $W_{n + 1}$,
even though $C_{n + 1}$ depends on the choice of a $p$-basis.
It makes the analysis of $h_{r}^{A}$ much similar
to the analysis of the Greenberg algebra in the perfect residue field case.
We need several facts about $C_{n + 1}$.

\begin{proposition}[{\cite[Section 2.1]{Sch72}}] \label{prop: presentation of Cohen ring scheme}
	Any element of $C_{n + 1}(Q)$ can be uniquely written as
		\begin{align*}
			&
					\sum_{0 \le i(1), \dots, i(d) \le p^{n} - 1}
						\bigl( x_{r}(i(1), \dots, i(d)), 0, \dots, 0 \bigr)
						\bigl( 1 \otimes t(1)^{i(1)} \cdots t(d)^{i(d)}, 0, \dots, 0 \bigr)
			\\
			&	\qquad	+ \dots +
			\\
			&		\sum_{0 \le i(1), \dots, i(d) \le p - 1}
						\bigl( 0, \dots, 0, x_{1}(i(1), \dots, i(d)), 0 \bigr)
						\bigl( 1 \otimes t(1)^{i(1)} \cdots t(d)^{i(d)}, 0, \dots, 0 \bigr)
			\\
			&	\qquad	+
					\bigl( 0, \dots, 0, x_{0} \bigr),
		\end{align*}
	where $x_{j}(i(1), \dots, i(d)) \in Q$
	 (and $x_{0} = x_{0}(0, \dots, 0) \in Q$).
	In particular, the functor of sets $C_{n + 1}$ (forgetting the ring structure) is represented by an affine space over $k$.
\end{proposition}

\begin{proposition}[{\cite[Sections 2.8.1, 2.10.1]{Sch72}}] \label{prop: Ver for Cohen ring scheme}
	Let $0 \le m \le n$ be integers.
	For $0 \le j \le m$, the collection of maps
		\begin{align*}
			&
					(0, \dots, 0, x, 0, \dots, 0)
					\bigl( 1 \otimes t(1)^{i(1)} \cdots t(d)^{i(d)}, 0, \dots, 0 \bigr)
			\\
			&	\mapsto
					\bigl(
						\Ver^{n - m}
						(0, \dots, 0, x, 0, \dots, 0)
					\bigr)
					\bigl( 1 \otimes t(1)^{i(1)} \cdots t(d)^{i(d)}, 0, \dots, 0 \bigr)
		\end{align*}
	with $x$ placed in the $j$-the position
	and $0 \le i(1), \dots, i(d) \le p^{m - j} - 1$
    defines a monomorphism of additive group schemes
	$C_{m + 1} \into C_{n + 1}$.
	The quotient $C_{n + 1} / C_{m + 1}$ is isomorphic to a finite product of copies of $C_{n - m}$.
\end{proposition}

\begin{proposition} \label{prop: mod p of Cohen ring scheme}
	Let $0 \le m \le n$ be integers.
	The image (in the fppf topology) of the endomorphism of $C_{n + 1}$ given by multiplication by $p^{n - m}$ agrees with the image of the morphism $C_{m + 1} \into C_{n + 1}$.
	In particular, the cokernel of multiplication by $p^{n - m}$ on $C_{n + 1}$	is isomorphic to a finite product of copies of $C_{n - m}$.
\end{proposition}

\begin{proof}
	With Propositions \ref{prop: presentation of Cohen ring scheme}
	 and \ref{prop: Ver for Cohen ring scheme},
	this follows from the facts that
	$p^{n - m} = \Ver^{n - m} \Fr^{n - m}$
	and the endomorphism $\Fr^{n - m}$ on $W_{n + 1}$ is an epimorphism.
\end{proof}

 Since $C_{1} = \Ga$, this proposition in particular implies that
$C_{n + 1}$ is a successive extension of copies of $\Ga$.
This means, in the terminology of \cite[Section B.1]{CGP15},
that $C_{n + 1}$ is a \emph{split} connected smooth unipotent group scheme over $k$.

More precisely, the quotient $C_{n + 1} / p^{n - m} C_{n + 1}$ with its induced ring scheme structure
is isomorphic to the Weil restriction of $C_{n - m}$
along $\Fr_{k}^{m + 1} \colon \Spec k \to \Spec k$.
We do not use this fact.

\begin{proposition} \label{prop: base change from Witt to Cohen}
	The natural morphism $W_{n + 1} \tensor_{W_{n + 1}(k)} C_{n + 1}(k) \to C_{n + 1}$
	is an isomorphism.
\end{proposition}

\begin{proof}
	We prove this by induction on $n$.
	The statement is obvious for $n = 0$.
	Assume the statement is true for some $n \ge 0$.
 We have an exact sequence
	\begin{equation} \label{eq: Kummer for Cohen ring scheme}
	    0 \to C_{n + 1} \to C_{n + 2} \to \Ga \otimes_{k, \Fr} k \to 0
	\end{equation}
	by \cite[Proposition 2.9]{Sch72} and it remains exact on $k$-sections.
	Consider the induced exact sequence
	\begin{equation} \label{eq: Kummer for Cohen tensored with Witt}
			W_{n + 2} \otimes_{W_{n + 2}(k)} C_{n + 1}(k)
		\to
			W_{n + 2} \otimes_{W_{n + 2}(k)} C_{n + 2}(k)
		\to
			W_{n + 2} \otimes_{W_{n + 2}(k), \Fr} k
		\to
			0.
	\end{equation}
	The natural morphisms define a morphism of exact sequences
	from \eqref{eq: Kummer for Cohen tensored with Witt}
	to \eqref{eq: Kummer for Cohen ring scheme}.
	We have
		\[
				W_{n + 2} \otimes_{W_{n + 2}(k)} C_{n + 1}(k)
			\isomto
				W_{n + 1} \otimes_{W_{n + 1}(k)} C_{n + 1}(k)
			\isomto
				C_{n + 1}
		\]
	by $W_{n + 2} / p^{n + 1} W_{n + 2} \cong W_{n + 1}$ and the induction hypothesis.
	Also,
		\[
				W_{n + 2} \otimes_{W_{n + 2}(k), \Fr} k
			\isomto
				\Ga \otimes_{k, \Fr} k.
		\]
	These imply the statement for $n + 1$ by the snake lemma.
\end{proof}

The $W_{r + 1}(k)$-algebra structure of $A$ lifts to a
$C_{r + 1}(k)$-algebra structure:

\begin{proposition}[{\cite[Theorem 8.4]{Sch72}}] \label{prop: Cohen lifted ghost component map}
	The ring homomorphism $W_{r + 1}(k) \to A$ given by the lifted $r$-th ghost component map uniquely extends to a ring homomorphism $C_{r + 1}(k) \to A$ such that each $(1 \otimes t(1)^{i(1)} \cdots t(d)^{i(d)}, 0 \dots, 0)$ is mapped to $\Tilde{t}(1)^{i(1)} \cdots \Tilde{t}(d)^{i(d)}$.
\end{proposition}

Now we can prove the representability of the $r$-twisted Greenberg algebra $h_{r}^{A}$:

\begin{proposition} \label{prop: canonical lift before RP is representable}
	The sheaf $h_{r}^{A}$ on $\Spec k_{\fppf}$ is represented by an affine space over $k$.
\end{proposition}

\begin{proof}
	By Proposition \ref{prop: base change from Witt to Cohen},
	we have $h_{r}^{A} \cong C_{r + 1} \otimes_{C_{r + 1}(k)} A$
	via the map $C_{r + 1}(k) \to A$ in
	Proposition \ref{prop: Cohen lifted ghost component map}.
	The ring $C_{r + 1}(k)$ is an artinian local ring with maximal ideal $(p)$ and residue field $k$
	by Proposition \ref{prop: Cohen ring is artinian with max ideal p}.
	The ring homomorphism $C_{r + 1}(k) \to A$ induces the identity map $k \to k$ on the residue fields.
	Hence, the ring $A$ as a $C_{r + 1}(k)$-module has finite length
	(\cite[Section 7.1, Lemma 1.36 (a)]{Liu02})
	and thus can be written as a finite direct sum of copies of
	$C_{r + 1}(k) / p^{s + 1} (C_{r + 1}(k))$
	for various $0 \le s \le r$.%
	\footnote{This is a general fact about finite length modules
	over truncated discrete valuation rings.
	In our case, it is enough to use the fact that
	$C_{r + 1}(k)$ is a quotient ring of the discrete valuation ring
	$C(k) = \invlim C_{n + 1}(k)$ with prime ideal $(p)$
	by \cite[Proposition 3.2.3]{Sch72}.}
	Therefore it is enough to show that
	$C_{r + 1} \otimes_{C_{r + 1}(k)} C_{r + 1}(k) / p^{s + 1} (C_{r + 1}(k))
	\cong C_{r + 1} / p^{s + 1} C_{r + 1}$
	is an affine space over $k$ for all $0 \le s \le r$.
	This sheaf is isomorphic to a finite product of copies of
	$C_{s + 1}$ by Proposition \ref{prop: mod p of Cohen ring scheme}.
	This proves the proposition.
\end{proof}

Using this, we will show the representability of $F_{X}^{r}$.
We need some preparations about patching.

\begin{proposition} \label{prop: canonical lift before RP commutes with localization}
	Let $Q$ be a $k$-algebra.
	Let $f \in Q$ be any element.
	Denote by $Q_{f} = Q[1 / f]$ the localization of $Q$ by $f$.
	Then the natural morphism $h_{r}^{A}(Q) \to h_{r}^{A}(Q_{f})$ induces an isomorphism
	$h_{r}^{A}(Q)_{(f, 0, \dots, 0) \otimes 1} \isomto h_{r}^{A}(Q_{f})$.
\end{proposition}

\begin{proof}
	From the proof of Proposition \ref{prop: canonical lift before RP is representable}, it is enough to show that $C_{s + 1}(Q)_{(f, 0, \dots, 0)} \isomto C_{s + 1}(Q_{f})$ for all $0 \le s \le r$.
	This follows from $W_{s + 1}(Q)_{(f, 0, \dots, 0)} \isomto W_{s + 1}(Q_{f})$ \cite[Chapter 0, (1.5.3)]{Ill79}.
\end{proof}

\begin{proposition} \label{prop: Greenberg before RP is a Zariski sheaf}
	For any $A$-scheme $X$,	the functor $F_{X}^{r} = X \compose h_{r}^{A}$ is a sheaf for the Zariski topology.
\end{proposition}

\begin{proof}
	The statement means that for any $k$-algebra $Q$ and
	any family of elements $f_{1}, \dots, f_{n}$ generating the unit ideal of $Q$,
	the sequence
		\[
				X(h_{r}^{A}(Q))
			\to
				\prod_{i}
					X(h_{r}^{A}(Q_{f_{i}}))
			\rightrightarrows
				\prod_{i, j}
					X(h_{r}^{A}(Q_{f_{i} f_{j}}))
		\]
	is an equalizer.
	Set $B = h_{r}^{A}(Q)$ and $g_{i} = (f_{i}, 0, \dots, 0) \otimes 1 \in B$.
	Then the elements $g_{i}$ generate the unit ideal of $B$.
	Therefore, by Proposition \ref{prop: canonical lift before RP commutes with localization},
	the above sequence can be identified with the obvious equalizer sequence
		\[
				X(B)
			\to
				\prod_{i}
					X(B_{g_{i}})
			\rightrightarrows
				\prod_{i, j}
					X(B_{g_{i} g_{j}}).
		\]
\end{proof}

For any $k$-algebra $Q$, we have a natural surjection
$h_{r}^{A}(Q) \onto Q \otimes_{k, \Fr_{k}^{r}} k$.
Hence for any $A$-scheme $X$, we have a natural morphism
$F_{X}^{r} \to \Weil_{\Fr_{k}}^{r}(X)$ of functors over $k$.

\begin{proposition} \label{prop: etale morphisms and Greenberg before RP}
	For any formally \'etale morphism $X' \to X$ of $A$-schemes,
	the diagram
		\[
			\begin{CD}
					F_{X'}^{r}
				@>>>
					F_{X}^{r}
				\\
				@VVV
				@VVV
				\\
					\Weil_{\Fr_{k}}^{r}(X')
				@>>>
					\Weil_{\Fr_{k}}^{r}(X)
			\end{CD}
		\]
	is cartesian.
\end{proposition}

\begin{proof}
	The kernel of $h_{r + 1}^{A} \onto \Ga \otimes_{k, \Fr^{r}} k$ is nilpotent.
	The same proof as Proposition \ref{r.cartesian} works.
\end{proof}

\begin{proposition} \label{prop: Greenberg before RP preserves Zariski coverings}
	For any open immersion $X' \into X$ of $A$-schemes, the morphism $F_{X'}^{r} \to F_{X}^{r}$ is relatively represented by open immersions of $k$-schemes.
	For any open covering $\{X_{\lambda}\}$ of $X$,
	the Zariski sheaves $\{F_{X_{\lambda}}^{r}\}$ form
	a Zariski covering of $F_{X}^{r}$.
\end{proposition}

\begin{proof}
	This follows from Proposition \ref{prop: etale morphisms and Greenberg before RP} and \cite[Lemma 1.5]{Kat86}.
\end{proof}

\begin{theorem}
	For any $A$-scheme $X$, the functor $F_{X}^{r}$ is representable by a $k$-scheme.
\end{theorem}

\begin{proof}
	If $X = \Affine_{A}^{1}$, then $F_{X}^{r} = h_{r}^{A}$,
	which is representable by
	Proposition \ref{prop: canonical lift before RP is representable}.
	The general case follows from
	Propositions \ref{prop: Greenberg before RP is a Zariski sheaf} and
	\ref{prop: Greenberg before RP preserves Zariski coverings}
	by the arguments similar to the proof of the representability of
	the relatively perfect Greenberg functor in Proposition \ref{p.gr}.
\end{proof}

\begin{definition}
	Let $X$ be an $A$-scheme.
	We call the object representing the functor $F_{X}^{r}$
	the \emph{$r$-twisted Greenberg transform} of $X$
	and denote it by $\Gr_{A}^{r}(X)$.
\end{definition}

When $k$ is perfect, the morphism
$\ff^{r} \otimes \id \colon h_{r}^{A} \to h^{A}$
mentioned after Proposition \ref{prop: Greenberg algebra indep of N}
induces a morphism of functors
$\Gr_{A}^{r} \to \Gr_{A}$
to the usual Greenberg functor.
This is not an isomorphism unless $r = 0$ and $A = k$.

Nevertheless, we will show that
the relative perfection of $\Gr_{A}^{r}(X)$ is
the relatively perfect Greenberg transform.
For this, we need to deal with the fppf sheafification in the definition of $h_{r}^{A}$,
which did not appear in the definition of the Kato canonical lifting $h^{A}$.
The following is needed instead of Proposition \ref{prop: mod p of Cohen ring scheme}:

\begin{proposition} \label{prop: p times Cohen ring scheme for RP algebra}
	Let $Q$ be a relatively perfect $k$-algebra.
	Let $0 \le m \le n$ be integers.
	Then $p^{n - m}(C_{n + 1}(Q)) = C_{m + 1}(Q)$ in $C_{n + 1}(Q)$,
	where we identified $C_{m + 1}(Q)$ with its image in $C_{n + 1}(Q)$
	using Proposition \ref{prop: Ver for Cohen ring scheme}.
\end{proposition}

\begin{proof}
	It is enough to treat the case $m = n - 1$.
	One inclusion $p(C_{n + 1}(Q)) \subseteq C_{n}(Q)$ is obvious.
	For the opposite inclusion,
	it is enough to show that any element of the form
	$(0, \dots, 0, x, 0, \dots, 0)$ with $x \in Q$ placed in the $j$-th position, $1 \le j \le n$, is in $p(C_{n + 1}(Q))$.
	 Here we number entries of Witt vectors in $W_{n + 1}$
	from the $0$-th position to the $n$-th position. 
	Since $Q$ is relatively perfect
	and hence
	$\ff_{Q} \otimes \mathrm{can} \colon Q \otimes_{k, \ff_{k}} k \isomto Q$,
	we can write $x$ as a sum of elements of the form
	$y^{p} t(1)^{i(1)} \cdots t(d)^{i(d)}$ with
	$y \in Q$ and $0 \le i(1), \dots, i(d) \le p - 1$.
	It is enough to treat the case where $x$ is a monomial:
	$x = y^{p} t(1)^{i(1)} \cdots t(d)^{i(d)}$.
	Then $(0, \dots, 0, x, 0, \dots, 0)$ is $p$ times $(0, \dots, 0, y, 0, \dots, 0)$ (with $y$ in the $(j - 1)$-st position) times $\bigl( 1 \otimes t(1)^{i(1) p^{n - j}} \cdots t(d)^{i(d) p^{n - j}},
	0, \dots, 0 \bigr)$.
	 Indeed,
	    \begin{align*}
    	    &
    	            p \cdot (0, \dots, 0, y, 0, \dots, 0) \cdot
    	            \bigl( 1 \otimes t(1)^{i(1) p^{n - j}} \cdots t(d)^{i(d) p^{n - j}},
    	            0, \dots, 0\bigr)
            \\
            &    =
    	            (0, \dots, 0, y^{p}, 0, \dots, 0) \cdot
    	            \bigl( 1 \otimes t(1)^{i(1) p^{n - j}} \cdots t(d)^{i(d) p^{n - j}},
    	            0, \dots, 0\bigr)
            \\
            &    =
                    \bigl( 0, \dots, 0,
                    y^{p} \otimes t(1)^{i(1) p^{n}} \cdots t(d)^{i(d) p^{n}},
                    0, \dots, 0 \bigr)
            \\
            &    =
                    \bigl( 0, \dots, 0,
                    y^{p} t(1)^{i(1)} \cdots t(d)^{i(d)},
                    0, \dots, 0 \bigr)
            \\
            &   =
                    (0, \dots, 0, x, 0, \dots, 0),
	    \end{align*}
    where the $y$ in the first line is in the $(j - 1)$-st position,
    the $y^{p}$ in the second line is in the $j$-th position,
    the non-zero entries in the rest of the lines are in the $j$-th position
    and the second equality follows from \cite[Chapter 0, (1.3.12)]{Ill79}.
\end{proof}

We also need the following instead of
Proposition \ref{prop: base change from Witt to Cohen}:

\begin{proposition} \label{prop: base change from Witt to Cohen for RP algebras}
    Let $Q$ be a relatively perfect $k$-algebra.
    Then the natural homomorphism
    $W_{n + 1}(Q) \otimes_{W_{n + 1}(k)} C_{n + 1}(k) \to C_{n + 1}(Q)$
    is an isomorphism.
    In particular, this isomorphism gives an isomorphism between the Kato canonical lifting $h^{A}(Q) = W_{r + 1}(Q) \otimes_{W_{r + 1}(k)} A$
    and $C_{r + 1}(Q) \otimes_{C_{r + 1}(k)} A$.
\end{proposition}

\begin{proof}
    The only difference from
    Proposition \ref{prop: base change from Witt to Cohen}
    is to show
        \begin{gather*}
                W_{n + 2}(Q) \otimes_{W_{n + 2}(k)} C_{n + 1}(k)
            \isomto
                W_{n + 1}(Q) \otimes_{W_{n + 1}(k)} C_{n + 1}(k),
            \\
                W_{n + 2}(Q) \otimes_{W_{n + 2}(k), \ff} k
                \isomto
            Q \otimes_{k, \ff} k.
        \end{gather*}
    By Proposition \ref{p.wh} \eqref{i.wr}, we have
    $W_{n + 2}(Q) \otimes_{W_{n + 2}(k)} W_{n + 1}(k)
    \cong W_{n + 1}(Q)$.
    This implies the first isomorphism.
    The second isomorphism follows similarly.
\end{proof}

\begin{proposition} \label{prop: RPn of twisted Greenberg is RP Greenberg}
	Let $X$ be an $A$-scheme.
	Then the relative perfection of $\Gr_{A}^{r}(X)$ is canonically isomorphic to $\Gr_{A}^{\RP}(X)$.
\end{proposition}

\begin{proof}
    For any $k$-algebra $Q$,
    we have a natural $A$-algebra homomorphism
        \[
                W_{r + 1}(Q) \otimes_{W_{r + 1}(k)} A
            \to
                (W_{r + 1} \otimes_{W_{r + 1}(k)} A)(Q),
        \]
    where $\otimes$ in the target is the tensor product of fppf sheaves.%
    \footnote{To see this,
    set $F(Q) = W_{r + 1}(Q) \otimes_{W_{r + 1}(k)} A$.
    Then $F$ is a presheaf.
    Hence we have a canonical morphism $F \to \Tilde{F}$
    to the fppf sheafification of $F$.
    Thus we have a natural map $F(Q) \to \Tilde{F}(Q)$.
    This is the one stated above.}
    The target is $h_{r}^{A}(Q)$ by definition.
    If $Q$ is relatively perfect,
    then the source is the Kato canonical lifting $h^{A}(Q)$.
    Therefore we have an $A$-algebra homomorphism
    $h^{A}(Q) \to h_{r}^{A}(Q)$ for relatively perfect $Q$.
	It is enough to show that this homomorphism
	$h^{A}(Q) \to h_{r}^{A}(Q)$ is an isomorphism
	for any relatively perfect $k$-algebras $Q$.
	 This homomorphism can be written as
		\[
				C_{r +1}(Q) \otimes_{C_{r + 1}(k)} A
			\to
				(C_{r + 1} \otimes_{C_{r + 1}(k)} A)(Q)
		\]
	by Propositions \ref{prop: base change from Witt to Cohen}
	and \ref{prop: base change from Witt to Cohen for RP algebras}.
	More generally, for a finite length $C_{r + 1}(k)$-module $M$,
	we show that the morphism 
		\[
				C_{r +1}(Q) \otimes_{C_{r + 1}(k)} M
			\to
				(C_{r + 1} \otimes_{C_{r + 1}(k)} M)(Q)
		\]
	is an isomorphism, where $\otimes$ on the right-hand side is the tensor product of fppf sheaves.
	It is enough to treat the case 
	$M = C_{r + 1}(k) / p^{s + 1} (C_{r + 1}(k))$
	with $0 \le s \le r$.
	In this case, the above morphism is
	$C_{r + 1}(Q) / p^{s + 1} (C_{r + 1}(Q)) \to (C_{r + 1} / p^{s + 1} C_{r + 1})(Q)$.
	We have $p^{s + 1} (C_{r + 1}(Q)) = C_{r - s}(Q)$
	by Proposition \ref{prop: p times Cohen ring scheme for RP algebra}
	and $p^{s + 1} C_{r + 1} = C_{r - s}$
	by Proposition \ref{prop: mod p of Cohen ring scheme}.
	In the exact sequence
		\[
			0 \to C_{r - s} \to C_{r + 1} \to C_{r + 1} / C_{r - s} \to 0
		\]
	of fppf sheaves, the first term $C_{r - s}$ is a successive extension
	of copies of $\Ga$
    by Proposition \ref{prop: Ver for Cohen ring scheme}.
        Since $H^{1}(Q, \Ga) = 0$
        by \cite[Tag 03P2]{Sta20}, it follows that
	the sequence
		\[
			0 \to C_{r - s}(Q) \to C_{r + 1}(Q) \to (C_{r + 1} / C_{r - s})(Q) \to 0
		\]
	is exact.
	This gives the desired isomorphism.
\end{proof}

 In particular, we have
$\Gr_{A}^{r}(X)(k) \cong \Gr_{A}^{\RP}(X)(k) \cong X(A)$.
Also $h_{r}^{A}(k) \cong h^{A}(k) \cong A$
from the proof.

Now we consider the infinite level case.
Let $A$ be a complete noetherian local ring with maximal ideal $I$
and residue field $k$ such that $[k : k^{p}] < \infty$.
For each integer $r \ge 0$, we set $A_{r} = A / I^{r + 1}$, so $A_{0} = k$.
We give $A_{r}$ a $W_{r + 1}(k)$-algebra structure
given by the lifted $r$-th ghost component map.
Then the $r$-twisted Greenberg algebra
$h_{r}^{A_{r}} = W_{r + 1} \tensor_{W_{r + 1}(k)} A_{r}$ and
the $r$-twisted Greenberg functor $\Gr_{A_{r}}^{r}$ make sense.
We will take a certain limit of them in $r$.
But it turns out the resulting objects are automatically relatively perfect:

\begin{proposition} \label{prop: canonical lift in infinite level is automatically RP}
	Let $A$ be a complete noetherian local ring with maximal ideal $I$
	and residue field $k$ such that $[k : k^{p}] < \infty$.
	Consider the morphisms
		\[
				\cdots
			\to
				h_{2}^{A_{2}}
			\to
				h_{1}^{A_{1}}
			\to
				h_{0}^{k} = \Ga
		\]
	of ring schemes over $k$, where each morphism $h_{r}^{A_{r}} \to h_{r - 1}^{A_{r - 1}}$ is given by $\Fr \otimes \mathrm{can}$.
	Define $h_{\infty}^{A}$ to be the inverse limit of
	this system of affine schemes.
	Then $h_{\infty}^{A}$ is relatively perfect over $k$.
	For any relatively perfect $k$-algebra $Q$,
	the $A$-algebra $h_{\infty}^{A}(Q)$ is
	the Kato canonical lifting $h^{A}(Q)$ of $Q$ over $A$.
\end{proposition}

\begin{proof}
	We want to show that the relative Frobenius for $h_{\infty}^{A}$ is an isomorphism.
	Recall that $h_{r}^{A_{r}} = W_{r + 1} \otimes_{W_{r + 1}(k)} A_{r}$.
	For any $r \ge 0$,  by Lemma \ref{lem: Frob twist of Greenberg algebra} below,
	the Frobenius twist of $h_{r}^{A_{r}}$ is given by
	$W_{r + 1} \otimes_{\Fr, W_{r + 1}(k)} A_{r}$.
	The relative Frobenius for $h_{r}^{A_{r}}$ is given by
	$\Fr \otimes \id$.
	We have a morphism $\mathrm{can} \otimes \mathrm{can}$
	from $W_{r + 2} \otimes_{\Fr, W_{r + 2}(k)} A_{r + 1} = h_{r + 1}^{A_{r + 1}}$
	to $W_{r + 1} \otimes_{W_{r + 1}(k)} A_{r} = h_{r}^{A_{r}}$.
	This gives a morphism from the Frobenius twist of $h_{\infty}^{A}$ to $h_{\infty}^{A}$
	inverse to the relative Frobenius. 
	Therefore, $h_{\infty}^{A}$ is relatively perfect over $k$.
	
	 Let $Q$ be a relatively perfect $k$-algebra.
	We show that $h_{\infty}^{A}(Q) \cong h^{A}(Q)$.
	Since $\RP$ is a right adjoint, it commutes with inverse limits.
	Hence $h_{\infty}^{A} = (h_{\infty}^{A})^{\RP}$ is the inverse limit of
	$(h_{r}^{A_{r}})^{\RP}$ in $r$ with transition morphisms $\ff \tensor \id$.
	We have $(h_{r}^{A_{r}})^{\RP} \cong h^{A_{r}}$ by (the proof of)
	Proposition \ref{prop: RPn of twisted Greenberg is RP Greenberg}.
	In the presentation $h^{A_{r}} = W_{r + 1} \tensor_{W_{r + 1}(k)} A$,	the reduction map $h^{A_{r}}(Q) \onto Q$ is given by
	$\ff^{r} \tensor \id$.
	Therefore, the transition morphism
	$\ff \tensor \mathrm{can} \colon h^{A_{r + 1}}(Q) \to h^{A_{r}}(Q)$
	is compatible with the reduction maps to $Q$.
	By the uniqueness of the Kato canonical liftings
	$h^{A_{r + 1}}(Q), h^{A_{r}}(Q)$,
	we know that this transition morphism has to agree with the canonical map in Remark \ref{rem.cl} \eqref{i.basechange}.
	Taking the inverse limit in $n$, we know that $h_{\infty}^{A}(Q) \cong h^{A}(Q)$.
\end{proof}

\begin{lemma} \label{lem: Frob twist of Greenberg algebra}
	For any $r \ge 0$, the Frobenius twist of $h_{r}^{A_{r}}$ is given by $W_{r + 1} \otimes_{\Fr, W_{r + 1}(k)} A_{r}$.
	The relative Frobenius for $h_{r}^{A_{r}}$ is given by
	$\Fr \otimes \id$.
\end{lemma}

\begin{proof}
    The absolute Frobenius $\ff_{k}$ induces an endomorphism of the site $\Spec k_{\fppf}$
    whose underlying functor on the underlying categories is
    $Y \mapsto Y^{(p)} = Y \times_{k, \ff_{k}} k$.
    Therefore, the pullback functor $\ff_{k}^{\ast}$ agrees with
    $Y \mapsto Y^{(p)}$ on representable sheaves
    and commutes with tensor products of sheaves (being a left adjoint).
    It is the identity functor on constant sheaves
    (which are disjoint unions of copies of $\Spec k$).
    Therefore the Frobenius twist of $h_{r}^{A_{r}}$ is
    $W_{r + 1}^{(p)} \tensor_{\mathrm{can}^{(p)}, W_{r + 1}(k)} A_{r}$.
    The relative Frobenius for a constant sheaf is the identity.
    Hence the relative Frobenius for $h_{r}^{A_{r}}$ is given by
    $\ff_{W_{r + 1} / k} \tensor \id$.
    The naturality of relative Frobenius gives a commutative diagram
        \[
            \begin{CD}
                    W_{r + 1}(k)
                @> \mathrm{can} >>
                    W_{r + 1}
                \\ @| @VV \ff_{W_{r + 1} / k} V \\
                    W_{r + 1}(k)
                @> \mathrm{can}^{(p)} >>
                    W_{r + 1}^{(p)}
            \end{CD}
        \]
     of $k$-schemes.
    The right vertical arrow followed by the natural  $k$-scheme isomorphism
    $W_{r + 1}^{(p)} \cong W_{r + 1}$ is
    the Frobenius endomorphism $\ff$ of $W_{r + 1}$.
    This means that $\mathrm{can}^{(p)} \colon W_{r + 1}(k) \to W_{r + 1}^{(p)}$
    can be identified with $\ff \colon W_{r + 1}(k) \to W_{r + 1}$.
    This gives the desired description of the relative Frobenius for $h_{r}^{A_{r}}$.
\end{proof}

Therefore, even if $k$ is perfect and $A = W(k)$,
the scheme $h_{\infty}^{A}$ is the perfection of $W$ and not $W$.
There seem to be no other natural transition morphisms among the family $\{h_{r}^{A_{r}}\}$
but $\ff \tensor \mathrm{can}$.

\begin{proposition}
	Let $A, I, k$ be as in
	Proposition \ref{prop: canonical lift in infinite level is automatically RP}.
	Let $X$ be an $A$-scheme.
	Then the functor $F_{X}^{\infty} \colon Q \mapsto X(h_{\infty}^{A}(Q))$ on $k$-algebras $Q$
	is representable by the relatively perfect Greenberg transform $\Gr^{\RP}(X)$.
\end{proposition}

\begin{proof}
	The functor $F_{X}^{\infty}$ is the inverse limit of
	the functors $Q\mapsto X(h_{r}^{A_{r}}(Q)) $
	for $r \ge 0$ by
	the Bhatt-Gabber algebraization theorem
	already used in the proof of Proposition \ref{p.adjinf}.
	Hence $F_{X}^{\infty}$ is the inverse limit of the schemes $\Gr_{A}^{r}(X)$.
	The transition morphisms $\Gr_{A}^{r + 1}(X) \to \Gr_{A}^{r}(X)$ are affine by the same proof as Proposition \ref{p.gr}
	(using Proposition \ref{prop: etale morphisms and Greenberg before RP}
	in place of \ref{r.cartesian}).
	Hence $F_{X}^{\infty}$ is a scheme.
	
    We show that $F_{X}^{\infty}$ is relatively perfect over $k$.
	Let $Q$ be a $k$-algebra.
	The composite of the absolute Frobenius $\ff_{k} \colon k \to k$
	and the structure map $k \to Q$ gives a new $k$-algebra structure
	on the ring $Q$.
	We denote this new $k$-algebra by ${^{(p)} Q}$.
	 The Frobenius map $Q \to {^{(p)} Q}$ (on the underlying rings) is a $k$-algebra homomorphism.
	For any $k$-scheme $Z$, the relative Frobenius morphism $Z \to Z^{(p)}$ on $Q$-valued points is given by the map $Z(Q) \to Z({^{(p)} Q})$ 	induced by the Frobenius map $Q \to {^{(p)} Q}$.
	In particular, the relative Frobenius morphism
	$F_{X}^{\infty} \to (F_{X}^{\infty})^{(p)}$
	for $F_{X}^{\infty}$ on $Q$-valued points
	is given by the map $F_{X}^{\infty}(Q) \to F_{X}^{\infty}({^{(p)} Q})$
	induced by the Frobenius map $Q \to {^{(p)} Q}$.
	This map is written as
	$X(h_{\infty}^{A}(Q)) \to X(h_{\infty}^{A}({^{(p)} Q}))$.
	Since $h_{\infty}^{A}$ is relatively perfect,
	its relative Frobenius morphism evaluated on $Q$-valued points,
	$h_{\infty}^{A}(Q) \to h_{\infty}^{A}({^{(p)} Q})$,
	is an isomorphism of $A$-algebras.
	This implies that 
	$X(h_{\infty}^{A}(Q)) \isomto X(h_{\infty}^{A}({^{(p)} Q}))$.
	Thus, $F_{X}^{\infty}$ is relatively perfect over $k$.
	
	The functor $F_{X}^{\infty}$ takes the same values as $\Gr^{\RP}(X)$
	by Proposition \ref{prop: canonical lift in infinite level is automatically RP}.
	Hence they are isomorphic.
\end{proof}


\section{Relatively perfectly smooth group schemes}
\label{sec: Relatively perfectly smooth group schemes}
Let $k$ be a field of characteristic $p > 0$ such that $[k : k^{p}] < \infty$.
In this section, we provide some preliminary results for the next section on relatively perfect cycle maps.
We need a solid foundation on relative perfections of commutative smooth group schemes over $k$.
We make a convention:
\begin{itemize}
	\item[]
		\emph{For the rest of the paper, all group schemes are assumed \emph{commutative}}.
\end{itemize}

\begin{definition} \label{def: finiteness for relatively perfect schemes}
	Let $X$ be a $k$-scheme.
	\begin{enumerate}
		\item
			$X$ is said to be \emph{locally relatively perfectly of finite presentation}
			if it is Zariski locally isomorphic to the relative perfection of a $k$-scheme (locally) of finite type (or finite presentation).
		\item
			$X$ is said to be \emph{relatively perfectly of finite presentation}
			if it is locally relatively perfectly of finite presentation, quasi-compact and quasi-separated.
		\item
			$X$ is said to be \emph{relatively perfectly smooth}
			if it is Zariski locally isomorphic to the relative perfection of a smooth $k$-scheme.
	\end{enumerate}
\end{definition}

\begin{definition} \mbox{}
	\begin{enumerate}
		\item
			Define $\Spec k_{\RP} = (\RPSch / k)_{\et}$ to be
			the category of relatively perfect $k$-schemes endowed with the \'etale topology
			(\cite[Section 2]{Kat86}).
		\item
			Denote the category of relatively perfectly smooth $k$-schemes with $k$-scheme morphisms
			by $(\RPSSch / k)$.
			Define $\Spec k_{\RPS} = (\RPSSch / k)_{\et}$ to be
			the category $(\RPSSch / k)$ endowed with the \'etale topology (\cite[Section 2]{KS19}).
		\item
			Define $\Spec k_{\Et} = (\Sch / k)_{\et}$ to be
			the category of $k$-schemes endowed with the \'etale topology (which is the big \'etale site of $k$).
		\item
			The category of sheaves of abelian groups on $\Spec k_{\RP}$ is denoted by $\Ab(k_{\RP})$.
			Its $\Hom$ and $n$-th $\Ext$ functors are denoted by
			$\Hom_{k_{\RP}}$ and $\Ext_{k_{\RP}}^{n}$, respectively.
			We use similar notation for $\Spec k_{\RPS}$ and $\Spec k_{\Et}$.
	\end{enumerate}
\end{definition}

\begin{definition}
	Let $X$ be a $k$-scheme locally of finite type.
	Then it admits a maximal geometrically reduced closed subscheme by \cite[Lemma C.4.1]{CGP15}.
	We denote this subscheme by $X^{\natural} \subseteq X$.
\end{definition}

As mentioned in \cite[Lemma C.4.1]{CGP15},
if $X$ is a group scheme, then $X^{\natural}$ is a smooth group scheme.
Further, if $X$ is locally of finite type over $k$, $X^{\natural}$ is the maximal smooth closed $k$-subgroup scheme of $X$ and hence $(X/X^{\natural})^{\natural}=0$.

\begin{proposition} \label{prop: geometrically reduced part has the same RPn}
	Let $X$ be a $k$-scheme locally of finite type.
	Then the natural morphism $(X^{\natural})^{\RP} \to X^{\RP}$ is an isomorphism.
\end{proposition}

\begin{proof}
	A relatively perfect $k$-scheme is geometrically reduced by Remark \ref{rem: relative perfectness and Frobenius}
	\eqref{item: relatively perfect implies geom reduced}.
	Hence the schematic image of the morphism $X^{\RP} \to X$
 is contained in $X^{\natural}$.
	This gives a morphism $X^{\RP} \to (X^{\natural})^{\RP}$
	inverse to the morphism $(X^{\natural})^{\RP} \to X^{\RP}$.
\end{proof}

\begin{proposition} \label{prop: relatively perfectly of finite type means limit preserving} \mbox{}
	Let $\{X_{\lambda}\}$ be a filtered inverse system of
	quasi-compact quasi-separated relatively perfect $k$-schemes with affine transition morphisms.
	Set $X = \invlim X_{\lambda}$.
	Let $Y$ be a $k$-scheme locally relatively perfectly of finite presentation.
	Then
		\[
				\dirlim_{\lambda} \Hom_{k}(X_{\lambda}, Y)
			\isomto
				\Hom_{k}(X, Y).
		\]
\end{proposition}

\begin{proof}
	The statement is Zariski local on $Y$ by a standard (but non-trivial) patching argument.
	Hence, we may assume that $Y$ is the relative perfection of a $k$-scheme $Y_{0}$ (locally) of finite type.
	We have $\Hom_{k}(X, Y) \isomto \Hom_{k}(X, Y_{0})$ and
	$\Hom_{k}(X_{\lambda}, Y) \isomto \Hom_{k}(X_{\lambda}, Y_{0})$.
	Hence, the statement reduces to the usual characterization of a $k$-scheme locally of finite type
(\cite[Proposition (8.14.2)]{Gro66}).
\end{proof}

\begin{proposition} \label{prop: kernel of RPn of smooth group is also}
	Let $G$ and $H$ be relative perfections of smooth group schemes over $k$.
	Let $\varphi \colon G \to H$ be a group scheme morphism over $k$.
	Then $\Ker(\varphi)$ is the relative perfection of a smooth group scheme over $k$.
\end{proposition}

\begin{proof}
	Write $\varphi$ as the relative perfection of a morphism
	$\varphi_{0} \colon G_{0} \to H_{0}$ between smooth group schemes.
	The relative perfection functor is left exact since it is a right adjoint.
	Hence $\Ker(\varphi_{0})^{\RP} \cong \Ker(\varphi)$.
	We have $\Ker(\varphi_{0})^{\RP} \cong (\Ker(\varphi_{0})^{\natural})^{\RP}$
	by Proposition \ref{prop: geometrically reduced part has the same RPn}.
	The group scheme $\Ker(\varphi_{0})^{\natural}$ is smooth,
	so this proves the proposition.
\end{proof}

Consider the following three types of $k$-group schemes:
\begin{enumerate}
    \item\label{RPft} group schemes relatively perfectly of finite presentation over $k$,
    \item\label{qcRPs} quasi-compact relatively perfectly smooth group schemes over $k$, and
    \item\label{RPqcs} relative perfections of quasi-compact smooth group schemes over $k$.
\end{enumerate}
Note that a group scheme over a field is separated (\cite[Tag 047L]{Sta20}) and hence quasi-separated.
Clearly \eqref{qcRPs} implies \eqref{RPft} and \eqref{RPqcs} implies \eqref{qcRPs} by definition and Remark \ref{rem: properties of RP} \eqref{RP preserves qc}. The next proposition states that \eqref{RPft} implies \eqref{RPqcs}  thus showing that the above three types of objects are the same thing.

\begin{proposition} \label{prop: RP finite type group is a RPn of smooth group}
	Let $G$ be a group scheme relatively perfectly of finite presentation over $k$.
	Then $G$ is the relative perfection of a quasi-compact smooth group scheme over $k$.
	Further, any morphism $\varphi\colon G\to H$ of group schemes relatively perfectly of finite presentation over $k$ is the relative perfection of a morphism $\varphi_0\colon G_0\to H_0$ of quasi-compact smooth group schemes over $k$.
\end{proposition}

\begin{proof}
	Since $G$ is a quasi-compact group scheme over $k$,
	it can be written as a filtered inverse limit of
	group schemes $G_{\lambda}$ of finite type over $k$
	with affine transition morphisms
	by \cite[Part 1, Chapter V, Section 3, Theorem 3.1]{Per75}.
	Applying the relative perfection functor,
	we have $G \cong \invlim G_{\lambda}^{\RP} \cong \invlim (G_{\lambda}^{\natural})^{\RP}$.
	Applying Proposition \ref{prop: relatively perfectly of finite type means limit preserving} to the identity map $G \to G$,
	we know that there exist an index $\lambda$ and a $k$-scheme morphism $(G_{\lambda}^{\natural})^{\RP} \to G$ such that the composite
	$G \to (G_{\lambda}^{\natural})^{\RP} \to G$ (where the first morphism is the canonical morphism of the inverse limit)
	is the identity.
	By increasing $\lambda$ if necessary,
	we can take such a morphism $(G_{\lambda}^{\natural})^{\RP} \to G$ to be a group scheme morphism over $k$.
	This means that $G$ is the kernel of an idempotent endomorphism of
	the group scheme $(G_{\lambda}^{\natural})^{\RP}$.
	Hence $G$ is the relative perfection of a quasi-compact smooth group scheme over $k$ by Proposition \ref{prop: kernel of RPn of smooth group is also}.
	The second statement follows from the first by the usual characterization of a $k$-scheme locally of finite type (\cite[Proposition (8.14.2)]{Gro66}).
\end{proof}

\begin{proposition} \label{prop: RPn of smooth surj is surj over the RP site}
    Let $f \colon Y \to X$ be a smooth surjective morphism of $k$-schemes.
    Then the induced morphism $f^{\RP} \colon Y^{\RP} \to X^{\RP}$ is
    an epimorphism of sheaves of sets on $\Spec k_{\RP}$.
\end{proposition}

\begin{proof}
    By \cite[Corollary (17.16.3) (ii)]{Gro67},
    there exist an \'etale surjective morphism $g \colon X' \to X$
    and an $X$-scheme morphism $s \colon X' \to Y$.
    They induce morphisms $g^{\RP} \colon X'^{\RP} \to X^{\RP}$ and
    $s^{\RP} \colon X'^{\RP} \to Y^{\RP}$.
    The \'etaleness of $g$ implies that
    $g^{\RP}$ can be obtained as the base change of $g$
    by the projection $X^{\RP} \to X$
    by \cite[Corollary 1.9]{Kat86}.
    Hence $g^{\RP}$ is \'etale surjective.
    In particular, it is an epimorphism of sheaves of sets on $\Spec k_{\RP}$.
    Hence so is $f^{\RP}$.
\end{proof}

\begin{proposition} \label{prop: cohomology of relative perfection}
	Let $X$ be a relatively perfect $k$-scheme.
	Let $G_{0}$ be a smooth group scheme over $k$
	with relative perfection $G = G_{0}^{\RP}$.
	Then for any $i \ge 0$, we have
		\[
				H^{i}(X_{\et}, G)
			\isomto
				H^{i}(X_{\et}, G_{0}).
		\]
	If $X$ is the relative perfection of a quasi-compact quasi-separated $k$-scheme $X_{0}$,
	then these isomorphic groups are further isomorphic to
		\[
				\dirlim_{n}
				H^{i}(\mathfrak{R}_{\ff_{k}}^{n}(X_{0})_{\et}, G_{0}).
		\]
\end{proposition}

\begin{proof}
     Since fiber products of relatively perfect $k$-schemes are relatively perfect
    by \cite[Lemma 1.2]{Kat86},
    the category $(\RPSch / k)$ has finite inverse limits
    and the inclusion functor $(\RPSch / k) \into (\Sch / k)$ commutes with them.
    Hence by \cite[Chapter II, Proposition 2.6 (a) and Section 3, the first paragraph]{Mil80},
    the inclusion functor $(\RPSch / k) \into (\Sch / k)$ defines
    a morphism of sites
	$f \colon \Spec k_{\Et} \to \Spec k_{\RP}$
	 (which in particular means that the pullback functor $f^{\ast}$ is exact).
	The pushforward functor $f_{\ast} \colon \Ab(k_{\Et}) \to \Ab(k_{\RP})$ is exact
	 by \cite[Chapter III, Proposition 3.1 (a)]{Mil80}.
	We have $f_{\ast} G_{0} = G_{0}^{\RP} = G$ by definition.
	These imply the isomorphism between $H^{i}(X_{\et}, G)$ and $H^{i}(X_{\et}, G_{0})$.
    From the construction of the relative perfection,
	we have $X = \invlim_{n} \Weil_{\ff_{k}}^{n}(X_{0})$.
	Hence the second statement about the commutativity with limits
	follows from \cite[Expos\'e VII, Corollaire 5.9]{AGV72}.
\end{proof}

\begin{proposition} \label{prop: extensions of relative perfection}
	Let $G_{0}$ and $H_{0}$ be quasi-compact smooth group schemes over $k$
	with relative perfections $G$ and $H$, respectively.
	Then for any $i \ge 0$, we have
		\[
				\Ext_{k_{\RP}}^{i}(G, H)
			\cong
				\Ext_{k_{\Et}}^{i}(G, H_{0})
			\cong
				\dirlim_{n}
				\Ext_{k_{\Et}}^{i}(\mathfrak{R}_{\ff_{k}}^{n}(G_{0}), H_{0}).
		\]
	In particular, any extension of $G$ by $H$ in $\Ab(k_{\RP})$ is represented by a group scheme relatively perfectly of finite presentation over $k$.
\end{proposition}

\begin{proof}
	Let $f \colon \Spec k_{\Et} \to \Spec k_{\RP}$ be
	as in the proof of Proposition \ref{prop: cohomology of relative perfection}.
	Then $f^{\ast} G = G$ and $f_{\ast} H_{0} = H$.
	As $f^{\ast}$ and $f_{\ast}$ are exact functors,	this implies the isomorphism between $\Ext_{k_{\RP}}^{i}(G, H)$
	and $\Ext_{k_{\Et}}^{i}(G, H_{0})$.
	
	For the second isomorphism,
	we fix a choice of a Deligne-Scholze functorial resolution 
  	$M(A) = (\cdots \to M_{1}(A) \to M_{0}(A))$
	of an abelian group $A$ \cite[Theorem 4.5]{Sch19}.
	Each term $M_{j}(A)$ is a \emph{finite} direct sum of groups of the form $\Z[A^{m}]$
	(the free abelian group generated by the set $A^{m}$) for various $m$.
	Define a complex $M(G)$ in $\Ab(k_{\Et})$ to be the \'etale sheafification of the complex of presheaves
	$X \mapsto M(G(X))$ (where $X$ runs over $k$-schemes).
	This complex of sheaves is a resolution of $G$ since sheafification is exact.
	Set $G_{n} = \mathfrak{R}_{\mathbf{\mathrm{F}}_{k}}^{n}(G_{0})$,
	so that $G = \invlim_{n} G_{n}$.
	We have two hyper-$\Ext$ spectral sequences
		\begin{gather*}
					E_{1}^{i j}
				=
					\dirlim_{n}
					\Ext_{k_{\Et}}^{j}(M_{i}(G_{n}), H_{0})
				\Longrightarrow
					\dirlim_{n}
					\Ext_{k_{\Et}}^{i + j}(G_{n}, H_{0}),
			\\
					E_{1}^{i j}
				=
					\Ext_{k_{\Et}}^{j}(M_{i}(G), H_{0})
				\Longrightarrow
					\Ext_{k_{\Et}}^{i + j}(G, H_{0})
		\end{gather*}
	and a morphism of spectral sequences from the first one to the second one
	compatible with the $E_{\infty}$-terms.
	Each term $\Ext_{k_{\Et}}^{j}(M_{i}(G), H_{0})$ is a finite direct sum of
	groups of the form
		\[
				\Ext_{k_{\Et}}^{j}(\Z[G^{m}], H_{0})
			\cong
				H^{j}((G^{m})_{\et}, H_{0})
		\]
	for various $m$.
	We have
		\[
				\dirlim_{n}
				H^{j}((G_{n}^{m})_{\et}, H_{0})
			\isomto
				H^{j}((G^{m})_{\et}, H_{0})
		\]
	by Proposition \ref{prop: cohomology of relative perfection}.
	Hence the morphism of the above spectral sequences is an isomorphism.
	Therefore their $E_{\infty}$-terms are isomorphic.
\end{proof}

Before the next proposition with a long, technical proof,
the following example is instructive:

\begin{example}\label{ex.ga}
    Consider the exact sequence
    $0 \to \alpha_{p} \to \Ga \stackrel{\ff}{\to} \Ga \to 0$
    of affine group schemes over $k$,
    where this $\ff$ is the relative Frobenius for $\Ga$ over $k$.
    We will work out what happens
    by applying the relative perfection functor $\RP$.
    We have $\alpha_{p}^{\RP} = 0$
    since a relatively perfect $k$-algebra is (geometrically) reduced (Remark \ref{rem: relative perfectness and Frobenius}
    \eqref{item: relatively perfect implies geom reduced}).
    Therefore we have an exact sequence
    $0 \to \Ga^{\RP} \stackrel{ \ff^{\RP}}{\to} \Ga^{\RP} \to \Coker(\ff^{\RP}) \to 0$
    of affine group schemes over $k$.
     Here the morphism $\ff^{\RP}$ evaluated on
    a relatively perfect $k$-algebra $R$ is the $p$-th power map on $R$.
    To describe $\Coker(\ff^{\RP})$,
    take a $p$-basis $t(1), \dots, t(d)$ of $k$.
    Let $R$ be a relatively perfect $k$-algebra.
    Let
        \[
                x
            =
                \sum_{0 \le i(1), \dots, i(d) \le p - 1}
                    x(i(1), \dots, i(d))^{p}
                    t(1)^{i(1)} \cdots t(d)^{i(d)}
            \in
                R
        \]
    be an arbitrary element.
    Then the map $x \mapsto x(0, \dots, 0)$ is functorial in $R$
    and hence gives a group scheme morphism
    $\Ga^{\RP} \to \Ga^{\RP}$.
    It is left inverse to $\ff^{\RP} \colon \Ga^{\RP} \into \Ga^{\RP}$.
    Therefore $\ff^{\RP} \colon \Ga^{\RP} \into \Ga^{\RP}$ is
    a closed immersion onto a direct summand.
    This implies that $\Coker(\ff^{\RP})$ is isomorphic to
    $(\Ga^{p^{d} - 1})^{\RP}$
    and the morphism $\Ga^{\RP} \onto \Coker(\ff^{\RP})$ is given by
        \[
                x
            \mapsto
                \bigl(
                    x(i(1), \dots, i(d))
                \bigr)_{0 \le i(1), \dots, i(d) \le p - 1;  \text{ not all zero}}.
        \]
    In particular, $\ff^{\RP}$ is not an isomorphism
    (unless $k$ is perfect),
    while the relative Frobenius for $\Ga^{\RP}$ over $k$
    is an isomorphism by definition.

    Here the fact already used is
    $\Ga^{\RP} \cong (\Ga^{p^{d}})^{\RP} \cong (\Weil_{\ff_{k}}(\Ga))^{\RP}$.
    A canonical description of $\Coker(\ff^{\RP})$
    independent of the choice of a $p$-basis
    can be given as the kernel of the Cartier operator
    $C \colon \Omega_{k, d = 0}^{1} \to \Omega_{k}^{1}$;
    see \cite[(3.1.2), (3.1.3)]{Kat86}.
\end{example}

\begin{proposition} \label{prop: RPn of smooth groups forms an abelian category}
	The group schemes relatively perfectly of finite presentation over $k$ form an abelian subcategory of $\Ab(k_{\RP})$
	closed under extensions.
\end{proposition}

\begin{proof}
	Let $\varphi \colon G \to H$ be a morphism of
	group schemes relatively perfectly of finite presentation over $k$, with kernel $K$, cokernel $L$ and image $M$ in $\Ab(k_{\RP})$.
	Then $K$ is relatively perfectly of finite presentation over $k$
	by Propositions \ref{prop: kernel of RPn of smooth group is also} and \ref{prop: RP finite type group is a RPn of smooth group}.
	We want to show that $L$ is a group scheme relatively perfectly of finite presentation over $k$.
	By  Proposition   \ref{prop: RP finite type group is a RPn of smooth group},
	we can write $\varphi$ as the relative perfection of a morphism
	$\varphi_{0} \colon G_{0} \to H_{0}$ of quasi-compact smooth group schemes over $k$.
	Set $G_{n} = \mathfrak{R}_{\ff_{k}}^{n}(G_{0})$ and
	$H_{n} = \mathfrak{R}_{\ff_{k}}^{n}(H_{0})$ for every $n$;
	they are smooth quasi-compact by \cite[Section 7.6, Proposition 5]{BLR90}.
	Denote by $K_{n}$, $L_{n}$ and $M_{n}$
	the kernel, cokernel and image, respectively,
	of the morphism of finite type
	$\varphi_{n} \colon G_{n} \to H_{n}$ in $\Ab(k_{\fppf})$ induced by $\varphi_{0}$.
	Then $L_{n}$ and $M_{n}$ are quasi-compact smooth group schemes over $k$  by \cite[Chapter II, Section 5, Corollary 2.2]{DG70}
	and $K_{n}$ is a group scheme of finite type over $k$.
	Note that the functors $\mathfrak{R}_{\ff_{k}}^{n}$ and $\RP$
	applied to group schemes of finite type over $k$ are only left exact,
	so $K_{0}^{\RP}=(K_{0}^{\natural})^{\RP}=K$,
	but $\mathfrak{R}_{\ff_{k}}^{n}(M_{0}) \ne M_{n}$ and $M_{0}^{\RP} \ne M$ in general.
	
	We first treat the special case
	where $\varphi_{0} \colon G_{0} \to H_{0}$ is faithfully flat
	and $K_{0}^{\natural} = 0$
	(which contains Example \ref{ex.ga} as a further special case).
	We have $K_{0}^{\RP} = 0$ by
	Proposition \ref{prop: geometrically reduced part has the same RPn}.
	Hence the morphism $K_{n} \to K_{0}$ is zero for large enough $n$.
	Choose such $n$.
	We have a diagram with exact rows
		\[
			\begin{CD}
					0
				@>>>
					K_{n}
				@>>>
					G_{n}
				@>>>
					M_{n}
				@>>>
					0
				\\
				@. @VVV @VVV @VVV @.
				\\
					0
				@>>>
					K_{0}
				@>>>
					G_{0}
				@>>>
					H_{0}
				@>>>
					0
			\end{CD}
		\]
	in $\Ab(k_{\fppf})$.
	Since the left vertical morphism $K_{n} \to K_{0}$ is zero,
	there exists a unique morphism $M_{n} \to G_{0}$
	splitting the right square into two commutative triangles.
	Applying $\RP$ to the right square, we have a diagram
		\[
			\begin{CD}
					G
				@>>>
					M_{n}^{\RP}
				\\
				@| @VVV
				\\
					G
				@>>>
					H
			\end{CD}
		\]
	with a morphism $M_{n}^{\RP} \to G$ splitting the square
	into two commutative triangles.
	All the morphisms in this square are monomorphisms
	since $K_{0}^{\RP} = K_{n}^{\RP} = 0$ and the left exactness of $\RP$.
	Therefore $G = M_{n}^{\RP}$ in $H$.
	Consider the exact sequence
		\[
				0
			\to
				M_{n}
			\to
				H_{n}
			\to
				L_{n}
			\to
				0
		\]
	in $\Ab(k_{\fppf})$.
	Since $M_{n}$ is smooth, the morphism $H_{n} \to L_{n}$ is smooth.
	Therefore the morphism $H \to L_{n}^{\RP}$ is an epimorphism in $\Ab(k_{\RP})$
	by Proposition \ref{prop: RPn of smooth surj is surj over the RP site}.
	Thus we have an exact sequence
		\[
				0
			\to
				M_{n}^{\RP}
			\to
				H
			\to
				L_{n}^{\RP}
			\to
				0
		\]
	in $\Ab(k_{\RP})$.
	With $G = M_{n}^{\RP}$, this shows that
	$L = \Coker(\varphi)$ is the relative perfection of
	the quasi-compact smooth group scheme $L_{n}$ (for this choice of $n$).
	
	Now we treat the general case. 
	As before, by Proposition \ref{prop: RPn of smooth surj is surj over the RP site},
	the exact sequence
		\[
				0
			\to
				M_{0}
			\to
				H_{0}
			\to
				L_{0}
			\to
				0
		\]
	in $\Ab(k_{\fppf})$ induces an exact sequence
		\[
				0
			\to
				M_{0}^{\RP}
			\to
				H
			\to
				L_{0}^{\RP}
			\to
				0
		\]
	in $\Ab(k_{\RP})$.
	Comparing this with the obvious exact sequence
	$0 \to M \to H \to L \to 0$,
	we have an exact sequence
		\[
				0
			\to
				M
			\to
				M_{0}^{\RP}
			\to
				L
			\to
				L_{0}^{\RP}
			\to
				0
		\]
	in $\Ab(k_{\RP})$.
	Consider the exact sequence of quasi-compact smooth group schemes
		\[
				0
			\to
				K_{0}^{\natural}
			\to
				G_{0}
			\to
				G_{0} / K_{0}^{\natural}
			\to
				0.
		\]
	As before, by Proposition \ref{prop: RPn of smooth surj is surj over the RP site},
	applying $\RP$ yields an exact sequence in $\Ab(k_{\RP})$.
	Hence the relative perfection of $G_{0} / K_{0}^{\natural}$ is 
	$G / K \cong M$.
	Consider the exact sequence
		\[
				0
			\to
				K_{0} / K_{0}^{\natural}
			\to
				G_{0} / K_{0}^{\natural}
			\to
				M_{0}
			\to
				0
		\]
	in $\Ab(k_{\fppf})$.
	We have $(K_{0} / K_{0}^{\natural})^{\natural} = 0$.
	Therefore the morphism $G_{0} / K_{0}^{\natural} \onto M_{0}$ is
	an example of the special case treated above.
	It follows that the cokernel of its relative perfection $M \to M_{0}^{\RP}$ in $\Ab(k_{\RP})$
	is the relative perfection of a quasi-compact smooth group scheme.
	Hence the exact sequence
		\[
				0
			\to
				M_{0}^{\RP} / M
			\to
				L
			\to
				L_{0}^{\RP}
			\to
				0
		\]
	in $\Ab(k_{\RP})$ is a presentation of $L$ as an extension
	between relative perfections of quasi-compact smooth group schemes.
	Therefore Proposition \ref{prop: extensions of relative perfection} shows that
	$L$ is a group scheme relatively perfectly of finite presentation.
\end{proof}

\begin{proposition} \label{prop: components of relatively perfectly smooth schemes}
	Let $X$ be a relatively perfectly smooth $k$-scheme.
	Then $X$ is the disjoint union of irreducible components.
	If $X = X_{0}^{\RP}$ is the relative perfection of a smooth $k$-scheme $X_{0}$,
	then the morphism $X \to X_{0}$
	is faithfully flat and induces a bijection between the irreducible components.
\end{proposition}

Here ``disjoint union'' means scheme-theoretic (or topological) disjoint union and not just set-theoretic disjoint union.
In particular, the proposition claims that
the irreducible components of $X$ are \emph{open}.

\begin{proof}
	We may assume that $X$ is the relative perfection of an affine $k$-scheme $X_{0}$
	that is \'etale over some affine space $\mathbb{A}_{k}^{n}$.
	Then $\mathfrak{R}_{\ff_{k}}(X_{0})$ is given by
	$X_{0} \times_{\mathbb{A}_{k}^{n}} \mathfrak{R}_{\ff_{k}}(\mathbb{A}_{k}^{n})$  
 by Proposition \ref{prop: Weil restriction for F commutes with etale}.
	The morphism $\mathfrak{R}_{\ff_{k}}(\mathbb{A}_{k}^{n}) \to \mathbb{A}_{k}^{n}$
	is given by the composite
		\[
				\mathfrak{R}_{\ff_{k}}(\mathbb{A}_{k}^{n})
			\stackrel{
			    \ff_{\mathfrak{R}_{\ff_{k}}(\mathbb{A}_{k}^{n}) / k}
			}{
			    \longrightarrow
			}
				\mathfrak{R}_{\ff_{k}}(\mathbb{A}_{k}^{n})^{(p)}
			\stackrel{\rho_{\mathbb{A}_{k}^{n}}}{\longrightarrow}
				\mathbb{A}_{k}^{n};
		\]
	see the paragraph after Lemma \ref{l.FrWeil}.
	The first morphism is a finite flat universal homeomorphism.
	An explicit calculation shows that the second morphism is
	a surjective $k$-linear map of finite-dimensional $k$-vector groups.   
	Therefore  their composite
	$\mathfrak{R}_{\ff_{k}}(\mathbb{A}_{k}^{n}) \to \mathbb{A}_{k}^{n}$ is
	faithfully flat with geometrically irreducible fibers.
	So is $\mathfrak{R}_{\ff_{k}}(X_{0}) \to X_{0}$.
	Hence this morphism induces a bijection between the irreducible components.
	Since the smooth $k$-scheme $X_{0}$ is the disjoint union of irreducible components
	by \cite[Tag 0357]{Sta20}, it follows that
	$\mathfrak{R}_{\ff_{k}}(X_{0})$ is the disjoint union of irreducible components.
	Iterating this for $\mathfrak{R}_{\ff_{k}}^{n}(X_{0})$ for all $n$,
	we get the result. 
\end{proof}

\begin{proposition} \label{prop: etale component scheme}
    Let $X$ be a relatively perfectly smooth $k$-scheme.
    Let $k^{\sep}$ be a separable closure of $k$.
    Then the natural action of $\Gal(k^{\sep} / k)$ on
    the set of connected (or irreducible) components $\pi_{0}(X_{k^{\sep}})$
    has open stabilizer group at each point.
    View $\pi_{0}(X_{k^{\sep}})$ as an \'etale scheme over $k$
    via this action.
    Then the natural morphism $X_{k^{\sep}} \to \pi_{0}(X_{k^{\sep}})$ over $k^{\sep}$
    canonically descends to $k$,
    and the resulting morphism $X \to \pi_{0}(X_{k^{\sep}})$ over $k$ is
    faithfully flat.
\end{proposition}

\begin{proof}
    The statement about stabilizers reduces to
    the case where $X$ is quasi-compact (or even affine),
    in which case it follows from
    Proposition \ref{prop: components of relatively perfectly smooth schemes}
    and \cite[Tag 038E]{Sta20}.
    The statements about descent and faithful flatness are standard.
\end{proof}
As usual, we frequently write this \'etale scheme $\pi_{0}(X_{k^{\sep}})$
simply by $\pi_{0}(X)$ by abuse of notation.
The points of the underlying topological space of the scheme $\pi_{0}(X)$
bijectively correspond to the connected components of $X$.

The following shows that relatively perfectly smooth group schemes over $k$
and relative perfections of smooth group schemes over $k$ are the same thing.

\begin{proposition} \label{prop: RPS implies RPn of smooth, separable points are dense}
	Let $G$ be a relatively perfectly smooth group scheme over $k$.
	Then $G$ is the relative perfection of a smooth group scheme over $k$.
	The set of points of $G$ whose residue fields are finite separable extensions of $k$ is dense in $G$.
\end{proposition}

\begin{proof}
	Denote the identity component of $G$ by $G^{0}$.
 It is geometrically irreducible and quasi-compact
	by \cite[Tag 0B7R]{Sta20}.
	The \'etale $k$-scheme $\pi_{0}(G)$ in
	Proposition \ref{prop: etale component scheme}
	has a canonical group scheme structure
	such that $G \to \pi_{0}(G)$ is a group scheme morphism.
	Its kernel is $G^{0}$ since $G^{0}$ is geometrically irreducible.
	Therefore $\pi_{0}(G)$ is the (fpqc) group scheme quotient $G / G^{0}$.
	By Proposition \ref{prop: RP finite type group is a RPn of smooth group},
	the group $G^{0}$ is the relative perfection of a quasi-compact
	smooth group scheme $H$ over $k$,
	which has to be connected by Proposition \ref{prop: components of relatively perfectly smooth schemes}.
	The natural morphism $G^{0} \to H$ is affine faithfully flat
	by the proof of Proposition \ref{prop: components of relatively perfectly smooth schemes}.
	Let $N$ be its kernel.
	We have an exact sequence
		\[
				0
			\to
				H
			\to
				G / N
			\to
				\pi_{0}(G)
			\to
				0
		\]
	of smooth group schemes over $k$.
	Applying $\RP$ yields an exact sequence in $\Ab(k_{\RP})$
	since $H$ is smooth (Proposition \ref{prop: RPn of smooth surj is surj over the RP site}).
	Hence $G \cong (G / N)^{\RP}$,
	so $G$ is the relative perfection of
	the smooth group scheme $G_{0} := G / N$ over $k$.
	
	For the statement about density, we may assume that $k$ is separably closed.
	For every $n$, the group of rational points $\mathfrak{R}_{\ff_{k}}^{n}(G_{0})(k) \cong G(k)$
	is dense in $\mathfrak{R}_{\ff_{k}}^{n}(G_{0})$
	since $\mathfrak{R}_{\ff_{k}}^{n}(G_{0})$ is smooth.
	Taking the inverse limit in $n$, we know that $G(k)$ is also dense in $G$.
\end{proof}

The precomposition with the inclusion functor
$(\RPSSch / k) \into (\RPSch / k)$
defines an exact functor
$\Ab(k_{\RP}) \to \Ab(k_{\RPS})$.

\begin{proposition} \label{prop: Ext comparison from RP to RPS}
	Let $G$ and $H$ be relatively perfectly smooth group schemes over $k$.
	Let $j \ge 0$ be an integer.
	Consider the homomorphism
		\[
				\Ext_{k_{\RP}}^{j}(G, H)
			\to
				\Ext_{k_{\RPS}}^{j}(G, H)
		\]
	induced by the exact functor $\Ab(k_{\RP}) \to \Ab(k_{\RPS})$ above.
	This homomorphism is an isomorphism.
\end{proposition}

\begin{proof}
	Let $M(A)$ be a choice of a Deligne-Scholze functorial resolution of an abelian group $A$
	as in the proof of Proposition \ref{prop: extensions of relative perfection}.
	Define a complex $M(G)$ in $\Ab(k_{\RP})$ to be the \'etale sheafification of the complex of presheaves
	$X \mapsto M(G(X))$ (where $X$ runs over relatively perfect $k$-schemes),
	which is a resolution of $G$.
	By restriction, this also defines a resolution of $G$ in $\Ab(k_{\RPS})$.
	Consider the spectral sequences
		\begin{gather*}
					E_{1}^{i j}
				=
					\Ext_{k_{\RP}}^{j}(M_{i}(G), H)
				\Longrightarrow
					\Ext_{k_{\RP}}^{i + j}(G, H),
			\\
					E_{1}^{i j}
				=
					\Ext_{k_{\RPS}}^{j}(M_{i}(G), H)
				\Longrightarrow
					\Ext_{k_{\RPS}}^{i + j}(G, H).
		\end{gather*}
	The functor $\Ab(k_{\RP}) \to \Ab(k_{\RPS})$ defines a morphism of spectral sequences
	from the first one to the second one compatible with the $E_{\infty}$-terms.
	The $E_{1}$-terms are described by the cohomology of finite products of $G$ with coefficients in $H$
	over $\Spec k_{\RP}$ and over $\Spec k_{\RPS}$.
	For any $m \ge 0$, these sites give the same cohomology groups $H^{j}(G^{m}_{\et}, H)$.
	Hence this morphism of spectral sequences is an isomorphism,
	and the result follows.
\end{proof}

\begin{proposition} \label{prop: RPS groups in RPS site}
	The group schemes relatively perfectly of finite presentation over $k$
	form an abelian subcategory of $\Ab(k_{\RPS})$
	closed under extensions.
\end{proposition}

\begin{proof}
	This follows from Propositions \ref{prop: RPn of smooth groups forms an abelian category}
	and \ref{prop: Ext comparison from RP to RPS}.
\end{proof}


\section{Construction of the relatively perfect cycle class map}
\label{sec: Relatively perfect cycle class maps}

Let $K$ be a complete discrete valuation field of mixed characteristic $(0, p)$
with residue field $k$ such that $[k : k^{p}] < \infty$.
Denote its ring of integers by $\Order_{K}$ with maximal ideal $\ideal{p}_{K}$.
We denote by $\Gr_{n}^{\RP}$ and $\Gr^{\RP}$ the relatively perfect Greenberg functors with respect to $\Order_{K} / \ideal{p}_{K}^{n + 1} \onto k$ and $\Order_{K} \onto k$, respectively.
We denote the Kato canonical lifting functor $h^{\Order_{K}}$ simply by $h$.
As we saw at the end of Section \ref{s.cl}, for a relatively perfect $k$-algebra $Q$, the $\Order_{K}$-algebra $h(Q)$ is characterized
as a unique (up to unique isomorphism) complete flat $\Order_{K}$-algebra
with $h(Q) \otimes_{\Order_{K}} k \cong Q$.

\begin{proposition} \label{prop: properties of Gr} \mbox{}
	Let $X$ be an $\Order_{K}$-scheme locally of finite type.
	Let $n \ge 0$ be an integer.
	\begin{enumerate}
		\item \label{item: Gr is RP finte type}
			The $k$-scheme $\Gr_{n}^{\RP}(X)$ is locally relatively perfectly of finite presentation.
		\item \label{item: Gr for smooth scheme}
			Assume $X$ is smooth.
			Then $\Gr_{n}^{\RP}(X)$ is relatively perfectly smooth over $k$.
			Consider the morphism $\Gr_{n + 1}^{\RP}(X) \to \Gr_{n}^{\RP}(X)$.
			Zariski locally on $\Gr_{n}^{\RP}(X)$,
			the scheme $\Gr_{n + 1}^{\RP}(X)$ is the relative perfection
			of an affine space.
			The morphism $\Gr^{\RP}(X) \to \Gr_{n}^{\RP}(X)$ admits a section
			Zariski locally.
		\item \label{item: components of Gr}
			Assume $X$ is smooth.
			The natural morphisms induce bijections between the sets of (open) irreducible components of
			$\Gr^{\RP}(X)$, $\Gr_{n}^{\RP}(X)$,
			$(X \times_{\Order_{K}} k)^{\RP}$ and $X \times_{\Order_{K}} k$.
	\end{enumerate}
\end{proposition}

\begin{proof}
	\eqref{item: Gr is RP finte type}
	This follows from the proofs of
	Lemmas \ref{l.Aone} and \ref{l.GrAff} and Proposition \ref{p.gr}.
	
	\eqref{item: Gr for smooth scheme}
	This follows from Propositions \ref{p.SmSch}
	and \ref{prop: structure of infinite level Gr for smooth X}.
	
	\eqref{item: components of Gr}
	This follows from \eqref{item: Gr for smooth scheme}
	and Proposition \ref{prop: components of relatively perfectly smooth schemes}.
\end{proof}

\begin{definition}
	Let $G$ be a group scheme over $\Order_{K}$.
	Let $n \ge 1$ be an integer.
	Define $U^{n} \Gr^{\RP}(G)$ to be the kernel of the natural morphism
	$\Gr^{\RP}(G) \to \Gr_{n - 1}^{\RP}(G)$.
\end{definition}

This notation mimics the usual notation for the $n$-th principal units
$U_{K}^{n} = 1 + \ideal{p}_{K}^{n}$.

\begin{proposition} \label{prop: multiplying p on filtration of Gr}
	Let $G$ be a smooth group scheme over $\Order_{K}$.
	Let $e_{K}$ be the absolute ramification index of $K$.
	Let $n > e_{K} / (p - 1)$ be an integer.
	Then the multiplication-by-$p$ map on $\Gr^{\RP}(G)$ restricts to an isomorphism $U^{n} \Gr^{\RP}(G) \isomto U^{n + e_{K}} \Gr^{\RP}(G)$.
\end{proposition}

\begin{proof}
	Let $R$ be a relatively perfect $k$-algebra.
	We want to show that $p \colon U^{n} \Gr^{\RP}(G)(R) \isomto U^{n + e_{K}} \Gr^{\RP}(G)(R)$.
	Let $A = h(R)$ be the Kato canonical lifting of $R$ to $\Order_{K}$.
	Then $A$ is a complete flat $\Order_{K}$-algebra with
	$A / A \ideal{p}_{K} \cong R$.
	The group of $R$-valued points $U^{n} \Gr^{\RP}(G)(R)$
	is the kernel of the map $G(A) \onto G(A / A  \ideal{p}_{K}^{n})$,
	so what we want to show is
		\[
				p
			\colon
				\Ker \bigl( G(A) \onto G(A / A  \ideal{p}_{K}^{n}) \bigr)
			\isomto
				\Ker \bigl( G(A) \onto G(A / A  \ideal{p}_{K}^{n+e_{K}})  \bigr).
		\]
	If $R = k$ and hence $A = \Order_{K}$,
	then this statement is true by \cite[Part II, Chapter IV, Section 9, Theorem 4]{Ser06}.
	The same proof works for a general $R$. \end{proof}

\begin{proposition} \label{prop: Gr mod p power is RP algebraic group}
	Let $G$ be a smooth group scheme over $\Order_{K}$
	such that the group of geometric points of $\pi_{0}(G \times_{\Order_{K}} k)$ is finitely generated.
	Then for any integer $m \ge 0$,
	the sheaf $\Gr^{\RP}(G) / p^{m} \Gr^{\RP}(G) \in \Ab(k_{\RP})$ is
	an affine group scheme relatively perfectly of finite presentation over $k$.
\end{proposition}

\begin{proof}
    For $n' > e_{K} / (p - 1)$,
    the natural morphism
        \[
                \Gr^{\RP}(G) / p^{m} \Gr^{\RP}(G)
            \to
                \Gr_{n' + m e_{K} - 1}^{\RP}(G) / p^{m} \Gr_{n' + m e_{K} - 1}^{\RP}(G)
        \]
    in $\Ab(k_{\RP})$ is an isomorphism by
	Propositions \ref{prop: properties of Gr} \eqref{item: Gr for smooth scheme}
	and \ref{prop: multiplying p on filtration of Gr}.
	Write $n = n' + m e_{K} - 1$.
	We have $\pi_{0}(\Gr_{n}^{\RP}(G)) \isomto \pi_{0}(G \times_{\Order_{K}} k)$
	by Proposition \ref{prop: properties of Gr} \eqref{item: components of Gr}.
	Hence its group of geometric points is finitely generated and $\Gr_{n}^{\RP}(G)^{0} =\Gr_{n}^{\RP}(G^{0})$ is relatively perfectly of finite presentation over $k$.
	The connected-\'etale sequence for $\Gr_{n}^{\RP}(G)$ induces an exact sequence
		\begin{align*}
			&		\pi_{0}(\Gr_{n}^{\RP}(G))[p^{m}]
				\to
					\Gr_{n}^{\RP}(G)^{0} / p^{m} \Gr_{n}^{\RP}(G)^{0}
			\\
			&	\to
					\Gr_{n}^{\RP}(G) / p^{m} \Gr_{n}^{\RP}(G)
				\to
					\pi_{0}(\Gr_{n}^{\RP}(G)) / p^{m} \pi_{0}(\Gr_{n}^{\RP}(G))
				\to
					0
		\end{align*}
	in $\Ab(k_{\RP})$, where $[p^{m}]$ denotes the part killed by $p^{m}$.
	All the terms except $\Gr_{n}^{\RP}(G) / p^{m} \Gr_{n}^{\RP}(G)$ are
	already shown to be group schemes relatively perfectly of finite presentation over $k$.
	Hence so is $\Gr_{n}^{\RP}(G) / p^{m} \Gr_{n}^{\RP}(G)$
	by Proposition \ref{prop: RPn of smooth groups forms an abelian category}.
 By Proposition \ref{prop: RP finite type group is a RPn of smooth group},
this group is the relative perfection of a quasi-compact smooth group scheme $H_{0}$ over $k$:
        \[
                \Gr_{n}^{\RP}(G) / p^{m} \Gr_{n}^{\RP}(G)
            =
                H_{0}^{\RP}.
        \]
    Note, in particular, that $H_{0}^{\RP}$ is killed by $p^{m}$.
    Since the canonical map $H_{0}^{\RP}\to H_{0}$ is a faithfully flat morphism of $k$-group schemes by Remark \ref{rem: properties of RP} \eqref{RP and limits} and Proposition \ref{prop: components of relatively perfectly smooth schemes}, $H_{0}$ is also killed by $p^{m}$.
	A connected smooth group scheme over $k$ killed by a power of $p$
	is affine by \cite[Section 9.2, Theorem 1]{BLR90}.
        Hence the identity component of $H_{0}$ is affine.
        As $H_{0}$ is of finite type over $k$, it follows that
        $H_{0}$ itself is affine.
        Therefore its relative perfection
        $\Gr_{n}^{\RP}(G) / p^{m} \Gr_{n}^{\RP}(G)$
        is affine.
\end{proof}

For any $m \ge 0$, recall the nearby cycle functor
$R^{m} \Psi \colon \Ab(K_{\Et}) \to \Ab(k_{\RPS})$
from \cite[Section 3]{KS19}:

\begin{definition}
	Let $F \in \Ab(K_{\Et})$ be a sheaf.
	Let $m \ge 0$.
	Define $R^{m} \Psi F \in \Ab(k_{\RPS})$ to be
	the \'etale sheafification of the presheaf
		\[
			R \mapsto H^{m}(h(R)_{K}, F),
		\]
	where $R$ runs over relatively perfectly smooth $k$-algebras
	and $h(R)_{K} = h(R) \otimes_{\Order_{K}} K$
	is the generic fiber of the Kato canonical lifting $h(R)$.
	The functors $\{R^{m} \Psi\}_{m \ge 0}$ form a cohomological functor.
\end{definition}

Actually this is a proposition \cite[Corollary 3.3]{KS19}
rather than the definition itself in \cite[Section 3]{KS19}.
The above characterization is sufficient for this paper.

The following proves Theorem \ref{thm: RP cycle class map}
\eqref{item: representability of nearby cycle}.

\begin{proposition} \label{prop: nearby cycle is RP algebraic group}
	Let $N$ be a $p$-primary finite Galois module over $K$.
	For any $m \ge 0$, the sheaf $R^{m} \Psi N \in \Ab(k_{\RPS})$
	is  an affine group scheme relatively perfectly of finite presentation over $k$.
\end{proposition}

\begin{proof}
	Let $L / K$ be a finite Galois extension containing a primitive $p$-th root of unity $\zeta_{p}$
	such that $N$ has trivial Galois action over $L$.
	Let $M$ be a subextension of $L / K$ corresponding to a $p$-Sylow subgroup of $\Gal(L / K)$.
	We have $\zeta_{p} \in M$.
	Let $k'$ be the residue field of $M$,
	which is a finite, not necessarily separable extension of $k$.
	For a relatively perfect $k'$-algebra $R'$,
	let $h^{\Order_{M}}(R')$ be the Kato canonical lifting of $R'$ to $\Order_{M}$
	and set $h^{\Order_{M}}(R')_{M} = h^{\Order_{M}}(R') \otimes_{\Order_{M}} M$.
	Let $R^{m} \Psi_{M} \colon \Ab(M_{\Et}) \to \Ab(k'_{\RPS})$ be the nearby cycle functor for $M$.
	For any integer $q \ge 0$, let $\Omega_{k'}^{q} = \Omega_{k' / \F_{p}}^{q}$ be
	the module of $q$-th absolute K\"ahler differentials of $k'$.
	As $[k' : k'^{p}] < \infty$, this is a finite-dimensional vector space over $k'$.
	We view it as a relatively perfect group scheme over $k'$
	isomorphic to the relative perfection of a finite direct sum of the additive algebraic group over $k'$.
	Let $\Omega_{k', d = 0}^{q}$ be the kernel of the differential
	$d \colon \Omega_{k'}^{q} \to \Omega_{k'}^{q + 1}$.
	Let $C \colon \Omega_{k', d = 0}^{q} \to \Omega_{k'}^{q}$ be the Cartier operator
	(\cite[Chapter 0, Section 2.1]{Ill79})
	as a morphism of sheaves on $\Spec k'_{\RP}$
	(\cite[Section 3]{Kat86}).
	By restriction, it defines a morphism of sheaves on $\Spec k'_{\RPS}$,
	which we denote by the same symbol $C$.
	Let $\nu(q)_{k'}$ be the kernel of $C - 1 \colon \Omega_{k', d = 0}^{q} \to \Omega_{k'}^{q}$.
	It is an affine group scheme relatively perfectly of finite presentation over $k'$
	by Proposition \ref{prop: RPS groups in RPS site}.
	
	Since $\Gal(L / M)$ is a $p$-group and $N$ is $p$-primary,
	 \cite[Chapter IX, Section 4, Lemma 4]{Ser79}, (ii) $\Rightarrow$ (i),
	shows that if the part of $N$ killed by $p$ is non-zero,
	then it contains a copy of the trivial $\Gal(L / M)$-module $\Z / p \Z$.
	It follows that
	the $\Gal(L / M)$-module $N$ admits a filtration
	whose successive subquotients are isomorphic to the trivial Galois module $\Z / p \Z$.
	Since $\zeta_{p} \in M$, any Tate twist of $\Z / p \Z$ over $M$ is isomorphic to $\Z / p \Z$.
	By \cite[Proposition 6.1]{KS19},
	we know that $R^{m} \Psi_{M}(\Z / p \Z)$ admits a filtration
	where every successive subquotient is isomorphic to
	either $\Omega_{k'}^{q}$, $d \Omega_{k'}^{q}$ or $v(q)_{k'}$ for some $q$
	(and this $q$ depends on the subquotient).
	Hence by Proposition \ref{prop: RPS groups in RPS site},
	we know that $R^{m} \Psi_{M} N$ is an affine group scheme
	relatively perfectly of finite presentation over $k'$.
	
	Let $\Weil_{M / K} \colon \Ab(M_{\Et}) \to \Ab(K_{\Et})$
	and $\Weil_{k' / k} \colon \Ab(k'_{\RPS}) \to \Ab(k_{\RPS})$ be
	the Weil restriction functors.
	They are exact functors by \cite[Tags 03QP (2) and 03YX]{Sta20}.
	The sheaf $R^{m} \Psi_{K}(\Weil_{M / K}(N)) \in \Ab(k_{\RPS})$ is
	the \'etale sheafification of the presheaf
		\[
				R
			\mapsto
				H^{m}(h^{\Order_{K}}(R)_{K}, \Weil_{M / K}(N))
			\cong
				H^{m}(h^{\Order_{K}}(R)_{K} \otimes_{K} M, N),
		\]
	where $R$ runs over relatively perfectly smooth $k$-algebras.
	The functor $\Weil_{k' / k}$ commutes with \'etale sheafification
	by Lemma \ref{lem: Weil restriction commutes with etale sheafification} below.
	Hence the sheaf $\Weil_{k' / k}(R^{m} \Psi_{M}(N)) \in \Ab(k_{\RPS})$ is
	the \'etale sheafification of the presheaf
		\[
				R
			\mapsto
				H^{m}(h^{\Order_{M}}(R \otimes_{k} k')_{M}, N).
		\]
	These sheaves are canonically isomorphic
	since $h^{\Order_{K}}(R) \otimes_{\Order_{K}} \Order_{M}$
	is a complete flat $\Order_{M}$-algebra whose reduction is $R \otimes_{k} k'$
	and so gives the Kato canonical lifting $h^{\Order_{M}}(R \otimes_{k} k')$.
	Since $R^{m} \Psi_{M}(N)$ is  an affine  group scheme
	relatively perfectly of finite presentation over $k'$,
	its Weil restriction to $k$ is  an affine group scheme
	relatively perfectly of finite presentation over $k$.
	Therefore $R^{m} \Psi_{K}(\Weil_{M / K}(N))$ is  an affine  group scheme
	relatively perfectly of finite presentation over $k$.
	
	Since $[M : K]$ is prime to $p$ and $N$ is $p$-primary,
	we know that $N$ is a direct summand of $\Weil_{M / K}(N)$
	as $\Gal(L / K)$-modules or in $\Ab(K_{\Et})$.
	Hence $R^{m} \Psi_{K}(N)$ is a direct summand of $R^{m} \Psi_{K}(\Weil_{M / K}(N))$.
	  By Proposition  \ref{prop: RPS groups in RPS site}, we know that $R^{m} \Psi_{K}(N)$ is  an affine group scheme
	relatively perfectly of finite presentation over $k$.
\end{proof}

\begin{lemma} \label{lem: Weil restriction commutes with etale sheafification}
    Let $k' / k$ be a finite extension.
    \begin{enumerate}
        \item \label{item: descent of refinement of etale coverings}
            Let $R$ be a $k$-algebra.
            Let $S'$ be a faithfully flat \'etale $R \tensor_{k} k'$-algebra.
            Then there exist a faithfully flat \'etale $R$-algebra $S$
            and an $R \tensor_{k} k'$-algebra homomorphism
            $S' \to S \tensor_{k} k'$.
                  \item \label{item: Weil restriction commutes with etale sheafification}
            Let $\Weil_{k' / k} \colon \Ab(k'_{\RPS}) \to \Ab(k_{\RPS})$
            be the Weil restriction functor.
            Denote its restriction to the presheaf categories
            by the same symbols $\Weil_{k' / k}$.
            Denote the \'etale sheafification functor by $\mathsf{a}$.
            Let $P$ be a presheaf on $\Spec k'_{\RPS}$.
            Then the natural morphism 
            $\mathsf{a}(\Weil_{k' / k}(P)) \to \Weil_{k' / k}(\mathsf{a}(P))$
            is an isomorphism.
    \end{enumerate}
\end{lemma}

\begin{proof}
    \eqref{item: descent of refinement of etale coverings}
    Take $\Spec S = \Weil_{R \tensor_{k} k' / R}(\Spec S')$,
    which is faithfully flat \'etale over $R$
    by \cite[Proposition A.5.2 (4), Corollary A.5.4 (1)]{CGP15}.
    Then the counit of adjunction defining the Weil restriction
    gives an $R \tensor_{k} k'$-algebra
    homomorphism $S' \to S \tensor_{k} k'$.
   
    \eqref{item: Weil restriction commutes with etale sheafification}
    The \'etale sheafification functor $\mathsf{a}$ is given by
    the zeroth \'etale \v{C}ech cohomology presheaf functor $\Check{\underline{H}}^{0}$
    applied twice (\cite[Chapter III, Remark 2.2 (c)]{Mil80}).
    Hence it is enough to show that the natural morphism
    $\Check{\underline{H}}^{0}(\Weil_{k' / k}(P)) \to
    \Weil_{k' / k}(\Check{\underline{H}}^{0}(P))$
    is an isomorphism.
   Let $R$ be a relatively perfectly smooth $k$-algebra.
    The value of the presheaf $\Check{\underline{H}}^{0}(\Weil_{k' / k}(P))$
    at $R$ is given by
        \begin{equation} \label{eq: Cech of Weil}
            \dirlim_{S / R}
                \Ker \bigl(
                        P(S \tensor_{k} k')
                    \to
                        P((S \tensor_{R} S) \tensor_{k} k')
                \bigr),
        \end{equation}
    where $S$ runs through faithfully flat \'etale $R$-algebras
    (and we introduce an equivalence relation on the category of such algebras
    by mutual refinements as in \cite[Remark 2.2 (a)]{Mil80}).
    Note that $(S \tensor_{R} S) \tensor_{k} k'$ can also be written as
    $(S \tensor_{k} k') \tensor_{R \tensor_{k} k'} (S \tensor_{k} k')$.
    The value of the presheaf $\Weil_{k' / k}(\Check{\underline{H}}^{0}(P))$
    at $R$ is given by
        \begin{equation} \label{eq: Weil of Cech}
            \dirlim_{S' / R \tensor_{k} k'}
                \Ker \bigl(
                        P(S')
                    \to
                        P(S' \tensor_{R \tensor_{k} k'} S')
                \bigr),
        \end{equation}
    where $S'$ runs through faithfully flat \'etale $R \tensor_{k} k'$-algebras.
    We have a natural homomorphism from \eqref{eq: Cech of Weil}
    to \eqref{eq: Weil of Cech} given by
    the map $S \mapsto S \tensor_{k} k'$ on the index sets.
    This map is cofinal by
    Statement \eqref{item: descent of refinement of etale coverings}
    (after the equivalence relation indicated above).
    Therefore this homomorphism is an isomorphism.
    Thus 
    $\Check{\underline{H}}^{0}(\Weil_{k' / k}(P)) \to
    \Weil_{k' / k}(\Check{\underline{H}}^{0}(P))$
    is an isomorphism.
\end{proof}

For a group scheme $\mathcal{G}$ over $\Order_{K}$,
the scheme $\Gr^{\RP}(\mathcal{G})$ as a functor in relatively perfect $k$-algebras $R$
is given by $R \mapsto \mathcal{G}(h(R))$.
The sheaf
$R^{0} \Psi(\mathcal{G} \times_{\Order_{K}} K) \in \Ab(k_{\RPS})$
is the \'etale sheafification of the presheaf
$R \mapsto \mathcal{G}(h(R)_{K})$,
where $R$ runs over relatively perfectly smooth $k$-algebras.
Therefore we have a canonical morphism
	\[
			\Gr^{\RP}(\mathcal{G})
		\to
			R^{0} \Psi(\mathcal{G} \times_{\Order_{K}} K)
	\]
in $\Ab(k_{\RPS})$.

\begin{definition} \label{def: RP cycle class map mod a p power}
	Let $G$ be a semi-abelian variety over $K$ with N\'eron model $\mathcal{G}$.
	Let $m \ge 0$ be an integer.
	Consider the exact sequence
            \begin{equation} \label{eq: Kummer sequence}
                0 \to G[p^{m}] \to G \stackrel{p^{m}}{\to} G \to 0
            \end{equation}
	and the induced morphisms
		\[
				\Gr^{\RP}(\mathcal{G}) / p^{m} \Gr^{\RP}(\mathcal{G})
			\to
				(R^{0} \Psi G) / p^{m} (R^{0} \Psi G)
			\to
				R^{1} \Psi(G[p^{m}])
		\]
	in $\Ab(k_{\RPS})$.
	Define a morphism
		\[
				\cl_{p, m}
			\colon
				\Gr^{\RP}(\mathcal{G}) / p^{m} \Gr^{\RP}(\mathcal{G})
			\to
				R^{1} \Psi(G[p^{m}]).
		\]
	by the composite of the above two morphisms.
	We call it the \emph{relatively perfect cycle class map mod $p^{m}$} for $G$ (in degree $1$).
\end{definition}

Note that it is not clear whether $\Gr^{\RP}(\mathcal{G}) \to R^{0} \Psi G$ is an isomorphism or not.
Also note that the Bhatt-Gabber algebraization theorem
used in Section \ref{sec: Greenberg transform of infinite level}
is crucial here:
there is a priori no morphism from
the inverse limit $\invlim_{n} \mathcal{G}(h^{\Order_{K} / \ideal{p}_{K}^{n}}(R))$
to $G(h^{\Order_{K}}(R)_{K})$.
We identified the former with $\mathcal{G}(h^{\Order_{K}}(R))$ by the aforementioned theorem
and then used the morphism from $\mathcal{G}(h^{\Order_{K}}(R))$ to the latter.%
\footnote{This means that if we do not use the Bhatt-Gabber theorem
and instead use the Greenberg transform of the formal completion
of $\mathcal{G}$ along the special fiber,
namely $R \mapsto \invlim_{n} \mathcal{G}(h^{\Order_{K} / \ideal{p}_{K}^{n}}(R))$,
then the target of the cycle class map will have to be
the \'etale sheafification of some rigid-analytic version
$R \mapsto \text{``} H^{1} \bigl( (\operatorname{Spf} h(R))_{K}^{\mathrm{rig}}, G[p^{n}] \bigr) \text{''}$.
Note that the canonical lifting $h(R)$ is highly non-noetherian.
Therefore it would take much effort to make this precise
and prove the representability of this \'etale sheafification.}
The representability of $\Gr^{\RP}(G)$ was proved by using its inverse limit presentation
$\invlim_{n} \Gr_{n}^{\RP}(G)$.

\begin{proposition} \label{prop: Gr to nearby cycle}
	Let the notation be as in Definition \ref{def: RP cycle class map mod a p power}.
	\begin{enumerate}
		\item \label{item: cycle class map is injective}
			The morphism $\cl_{p, m}$ is a closed immersion between
			 affine group schemes relatively perfectly of finite presentation over $k$.
		\item \label{item: points of cokernel of cycle class map}
			Let $k^{\sep}$ be a separable closure of $k$.
			Let $K^{\ur}$ be the corresponding unramified extension of $K$
			with completion $\Hat{K}^{\ur}$.
			Then the $k^{\sep}$-valued points of the exact sequence
				\begin{equation} \label{eq: cycle class exact sequence}
						0
					\to
						\Gr^{\RP}(\mathcal{G}) / p^{m} \Gr^{\RP}(\mathcal{G})
					\stackrel{\cl_{p, m}}{\to}
						R^{1} \Psi(G[p^{m}])
					\to
						\Coker(\cl_{p, m})
					\to
						0
				\end{equation}
			is canonically isomorphic to the exact sequence
				\[
						0
					\to
						G(\Hat{K}^{\ur}) / p^{m} G(\Hat{K}^{\ur})
					\to
						H^{1}(\Hat{K}^{\ur}, G[p^{m}])
					\to
						H^{1}(\Hat{K}^{\ur}, G)[p^{m}]
					\to
						0
				\]
			coming from
                the sequence \eqref{eq: Kummer sequence}.
		\item \label{item: vanishing of H one implies cycle class map invertible}
			Assume that $H^{1}(\Hat{K}^{\ur}, G)[p^{\infty}] = 0$.
			Then the morphism $\cl_{p, m}$ is an isomorphism.
			This includes the case of $G = \Gm$.%
                \footnote{
                        A little more generally,
                        the case where $G$ is a torus split by a tamely ramified extension of $K$.
                        See also \cite[Corollary 4.1.10]{Bra04}.
                }
	\end{enumerate}
\end{proposition}

An equivalent way to phrase
Statement \eqref{item: points of cokernel of cycle class map}
is that the restriction of the sequence \eqref{eq: cycle class exact sequence}
to the small \'etale site of $k$
is the exact sequence
    \[
            0
        \to
            i^{\ast} j_{\ast}(G) / p^{m} i^{\ast} j_{\ast}(G)
        \to
            i^{\ast} R^{1} j_{\ast}(G[p^{m}])
        \to
            (i^{\ast} R^{1} j_{\ast}(G))[p^{m}]
        \to
            0
        \]
coming from \eqref{eq: Kummer sequence},
where $j \colon \Spec K_{\et} \to \Spec \Order_{K, \et}$
and $i \colon \Spec k_{\et} \to \Spec \Order_{K, \et}$
are the natural morphisms on the small \'etale sites.

\begin{proof}
	\eqref{item: cycle class map is injective}
	The group of geometric points of $\pi_{0}(\mathcal{G} \times_{\Order_{K}} k)$ is
	finitely generated   as proven for example in    \cite[Proposition 3.5]{HN11}). 
	Hence the sheaves $\Gr^{\RP}(\mathcal{G}) / p^{m} \Gr^{\RP}(\mathcal{G})$ and
	$R^{1} \Psi(G[p^{m}])$ are affine group schemes
	relatively perfectly of finite presentation over $k$
	by Proposition \ref{prop: Gr mod p power is RP algebraic group} and
	Proposition \ref{prop: nearby cycle is RP algebraic group}, respectively.
	Their sets of points whose residue fields are finite separable extensions of $k$ are dense
	by Proposition \ref{prop: RPS implies RPn of smooth, separable points are dense}.
	Therefore it is enough to show that the map is injective on $k'$-valued points
	for every finite separable $k' / k$.
 The map
		\begin{equation} \label{eq: finite separable points of cycle class map}
				\mathcal{G}(h(k')) / p^{m} \mathcal{G}(h(k'))
			\isomto
				G(h(k')_{K}) / p^{m} G(h(k')_{K})
			\to
				H^{1}(h(k')_{K}, G[p^{m}])
		\end{equation}
	is injective.
	Here the first isomorphism comes from the universal property of N\'eron models since $h(k')$ is \'etale over the henselian ring $h(k)=\cO_K$ by Remark \ref{rem.cl} \eqref{i.etale lifting}.
	Hence the result follows.
	
	\eqref{item: points of cokernel of cycle class map}
	Note that $h(k')_{K}$ for a finite separable $k' / k$ is
	the finite unramified extension of $K$ with residue extension $k' / k$.
	Hence the results follows from \eqref{eq: finite separable points of cycle class map}.
	
	\eqref{item: vanishing of H one implies cycle class map invertible}
	This follows from the previous two statements.
\end{proof}

This proposition, together with the next definition,
proves Theorem \ref{thm: RP cycle class map}
\eqref{item: existence of cycle class map}.
 Recall from Remark \ref{rem: properties of RP} \eqref{RP and limits}
that a filtered inverse limit of affine relatively perfect group schemes
is an affine relatively perfect group scheme.

\begin{definition}\label{def: RP p-adic cycle class map}
	Let $G$ be a semi-abelian variety over $K$ with N\'eron model $\mathcal{G}$.
	\begin{enumerate}
		\item
			The relatively perfect cycle class maps $\cl_{p, m}$ mod $p^{m}$
			form a direct system in $m$ via natural morphisms.
			Define a morphism
				\[
						\cl_{p, \infty}
					\colon
						\Gr^{\RP}(\mathcal{G}) \otimes_{\Z} \Q_{p} / \Z_{p}
					\into
						R^{1} \Psi(G[p^{\infty}])
				\]
			in $\Ab(k_{\RPS})$ to be its direct limit
			(where $[p^{\infty}]$ denotes the $p$-primary torsion part).
			We call it the \emph{relatively perfect cycle class map $\otimes \, \Q_{p} / \Z_{p}$}.
		\item
		     Define $R^{1} \Psi(T_{p} G)$ to be the inverse limit in $n$ of
		    the relatively perfect affine group schemes $R^{1} \Psi(G[p^{n}])$
		    over $k$.
		\item
			The morphisms $\cl_{p, m}$ also form an inverse system in $m$ via other natural morphisms.
			Define a morphism
				\[
						\cl_{p}^{\wedge}
					\colon
						\Gr^{\RP}(\mathcal{G})^{\wedge_{p}}
					\into
						R^{1} \Psi(T_{p} G)\]
			to    be the inverse limit as a morphism of  relatively perfect affine group schemes over $k$.
			We call it the \emph{relatively perfect $p$-adic cycle class map}.
		\item
        	 Define a morphism
        		\[
        				\cl_{p}
        			\colon
        				\Gr^{\RP}(\mathcal{G})
        			\to
        				R^{1} \Psi(T_{p} G)
        		\]
        	to be the composite of the natural morphism
	$\Gr^{\RP}(\mathcal{G}) \to \Gr^{\RP}(\mathcal{G})^{\wedge}$
        	and the morphism $\cl_{p}^{\wedge}$ above.
        	We call it the \emph{relatively perfect cycle class map} for $G$
        	(in degree $1$).
        	It is a morphism of relatively perfect group schemes over $k$. 
	\end{enumerate}
\end{definition}

Proposition \ref{prop: Gr to nearby cycle}
\eqref{item: points of cokernel of cycle class map}
proves Theorem \ref{thm: RP cycle class map}
\eqref{item: points of cycle class map}.
Proposition \ref{prop: Gr to nearby cycle}
\eqref{item: cycle class map is injective}, together with the following proposition,
proves Theorem \ref{thm: RP cycle class map}
\eqref{item: inj and isom for cycle class map}:

\begin{proposition}\label{prop: cycle class map for tori}
	Let $G = T$ be a torus over $K$ with N\'eron model $\mathcal{T}$.
	Then $\Coker(\cl_{p, m}) \to \Coker(\cl_{p, m + 1})$ is a closed immersion
	for all $m$
	and an isomorphism for all large $m$.%
        \footnote{
            The proof below shows that $m$ can be taken to be
            the exponent of the maximal power
            of $p$ that divides the ramification index of the minimal splitting
            field of $G$ over $K$.
            }
	In particular, $\Coker(\cl_{p, \infty})$ is a group scheme
	relatively perfectly of finite presentation over $k$
	and $\cl_{p}^{\wedge}$ is an isomorphism.
\end{proposition}

\begin{proof}
	It is enough to show that
	$\Coker(\cl_{p, m})(k^{\sep}) \to \Coker(\cl_{p, m + 1})(k^{\sep})$
	is an injection for all $m$ and an isomorphism for all large $m$.
	We have
		\[
				\Coker(\cl_{p, m})(k^{\sep})
			\cong
				H^{1}(\Hat{K}^{\ur}, T)[p^{m}]
		\]
	by Proposition \ref{prop: Gr to nearby cycle}
	\eqref{item: points of cokernel of cycle class map}.
	Hence the injectivity follows.
	For the invertibility for large $m$, we may assume that $k$ is separably closed
	and it is enough to show that $H^{1}(K, T)$ is killed by multiplication by some positive integer.
	Let $L / K$ be a finite Galois extension that trivializes $T$.
	Set $G = \Gal(L / K)$.
	The exact sequence
		\[
				0
			\to
				H^{1}(G, T(L))
			\to
				H^{1}(K, T)
			\to
				H^{1}(L, T)^{G}
		\]
	and the vanishing of $H^{1}(L, \Gm)$ shows that
	$H^{1}(G, T(L)) \isomto H^{1}(K, T)$.
	The group $H^{1}(G, T(L))$ is killed by multiplication by $[L : K]$.
	This proves the result.
\end{proof}

This finishes the proof of Theorem \ref{thm: RP cycle class map}.


\end{document}